\documentclass[12pt]{amsart}

\usepackage[OT2, T1]{fontenc}

\usepackage{amscd}
\usepackage{amsmath}
\usepackage{amssymb}
\usepackage{mathrsfs} 			
\usepackage{units}
\usepackage[all]{xy}

\usepackage{algorithm}
\usepackage{algorithmic}

\def\frk{\frak}               

\def\pp{{\frk p}}
\def\Pp{{\frk P}}
\def\qq{{\frk q}}

\def\mm{{\frk m}}

\def\Phi{{\frk n}}
\def\Phi{{\frk N}}
\def\opn#1#2{\def#1{\operatorname{#2}}} 
\opn\chara{char} \opn\length{\ell} \opn\pd{pd} \opn\rk{rk}
\opn\projdim{proj\,dim} \opn\injdim{inj\,dim} \opn\rank{rank}
\opn\depth{depth} \opn\sdepth{sdepth} \opn\fdepth{fdepth}
\opn\grade{grade} \opn\height{height} \opn\embdim{emb\,dim}
\opn\codim{codim}  \opn\min{min} \opn\max{max}

\opn\Tr{Tr} \opn\bigrank{big\,rank}
\opn\superheight{superheight}\opn\lcm{lcm}
\opn\trdeg{tr\,deg}
\opn\reg{reg} \opn\lreg{lreg} \opn\ini{in} \opn\lpd{lpd}
\opn\size{size}
\opn\div{div} \opn\Div{Div} \opn\cl{cl} \opn\Cl{Cl}
\opn\Spec{Spec} \opn\Supp{Supp} \opn\supp{supp} \opn\Sing{Sing}
\opn\Ass{Ass} \opn\Min{Min}
\opn\Ann{Ann} \opn\Rad{Rad} \opn\Soc{Soc}
\opn\Im{Im} \opn\Ker{Ker} \opn\Coker{Coker} \opn\Am{Am}
\opn\Hom{Hom} \opn\Tor{Tor} \opn\Ext{Ext} \opn\End{End}
\opn\Aut{Aut} \opn\id{id}  \opn\deg{deg}

\opn\nat{nat}
\opn\pff{pf}
\opn\Pf{Pf} \opn\GL{GL} \opn\SL{SL} \opn\mod{mod} \opn\ord{ord}
\opn\Gin{Gin} \opn\Hilb{Hilb}
\opn\aff{aff} \opn\con{conv} \opn\relint{relint} \opn\st{st}
\opn\lk{lk} \opn\cn{cn} \opn\core{core} \opn\vol{vol}
\opn\link{link} \opn\star{star}
\opn\gr{gr}

\def\pot#1#2{#1[\kern-0.28ex[#2]\kern-0.28ex]}

\opn\dirlim{\underrightarrow{\lim}}
\opn\inivlim{\underleftarrow{\lim}}
\let\tensor=\otimes

\let\to=\rightarrow

\def\Implies{\ifmmode\Longrightarrow \else
        \unskip${}\Longrightarrow{}$\ignorespaces\fi}
\def\implies{\ifmmode\Rightarrow \else
        \unskip${}\Rightarrow{}$\ignorespaces\fi}
\def\iff{\ifmmode\Longleftrightarrow \else
        \unskip${}\Longleftrightarrow{}$\ignorespaces\fi}

\let\:=\colon
\newtheorem{Theorem}{Theorem}[]
\newtheorem{Lemma}[Theorem]{Lemma}
\newtheorem{Corollary}[Theorem]{Corollary}
\newtheorem{Proposition}[Theorem]{Proposition}

\theoremstyle{definition}

\newtheorem{Remark}[Theorem]{Remark}

\newtheoremstyle{subsection-tweak}
   {11pt}
   {3pt}%
   {}
   {}%
   {\bfseries}
   {}%
   {.5em}
   {\thmnumber{\@{#1}{}\@{#2}.}%
    \thmnote{~{\bfseries#3.}}}    

\newcounter{numberingbase}

\theoremstyle{subsection-tweak}
\newtheorem{bpp}[Theorem]{}
\newtheorem{bppt}[numberingbase]{}
\newcommand{\bbpp}{\begin{bpp}}
\newcommand{\eepp}{\end{bpp}}
\newcommand{\bbppt}{\begin{bppt}}
\newcommand{\eeppt}{\end{bppt}}

\theoremstyle{theorem}
\newtheorem{TheoremT}[numberingbase]{Theorem}
\newtheorem{LemmaT}[numberingbase]{Lemma}

\newtheorem{PropositionT}[numberingbase]{Proposition}

\newtheorem{variantT}[numberingbase]{Variant}

\theoremstyle{definition}
\newtheorem{ExampleT}[numberingbase]{Example}

\newcommand{\val}{\mathrm{val}}		

\providecommand{\qxq}[1]{\quad\text{#1}\quad}
\providecommand{\qx}[1]{\quad\text{#1}}

\newcommand{\tst}{\textstyle}

\newcommand{\sU}{\mathscr{U}}
\newcommand{\msk}{\medskip}
\newcommand{\bsk}{\bigskip}
\DeclareMathOperator{\Frac}{Frac}		
\newcommand{\wt}{\widetilde}
\newcommand{\ssk}{\smallskip}
\providecommand{\up}[1]{{\upshape(}#1{\upshape)}}
\providecommand{\ucolon}{{\upshape:} }
\newcommand{\surjects}{\twoheadrightarrow}
\newcommand{\hra}{\hookrightarrow}

\let\epsilon\varepsilon
\let\phi=\varphi
\def\qed{\ifhmode\textqed\fi
      \ifmmode\ifinner\quad\qedsymbol\else\dispqed\fi\fi}
\def\textqed{\unskip\nobreak\penalty50
       \hskip2em\hbox{}\nobreak\hfil\qedsymbol
       \parfillskip=0pt \finalhyphendemerits=0}
\def\dispqed{\rlap{\qquad\qedsymbol}}

\opn\dis{dis}
\def\pnt{{\raise0.5mm\hbox{\large\bf.}}}

\opn\Lex{Lex}

\begin{document}

\title{N\'{e}ron desingularization of extensions of valuation rings
 \\ {\tiny with an Appendix by K\k{e}stutis \v{C}esnavi\v{c}ius}}

\author{Dorin Popescu}
\address{Simion Stoilow Institute of Mathematics of the Romanian Academy, Research unit 5, University of Bucharest, P.O. Box 1-764, Bucharest 014700, Romania}

\begin{abstract}
Zariski's local uniformization, a weak form of resolution of singularities, implies that every valuation ring containing  $\bf Q$ is a filtered direct limit of smooth $\bf Q$-algebras. Given an immediate extension of valuation rings $V\subset V'$ containing $\bf Q$ we show  that $V'$ is a filtered direct limit of smooth $V$-algebras. This corrects a paper of us \cite{Po1} where we thought  that we may reduce to the case when the value groups are finitely generated. For this correction we use   an infinite tower of ultrapowers construction that rests on results from model theory. 
\end{abstract}

\maketitle

\section{A version of local uniformization}

Zariski  proved in characteristic $0$ in \cite{Zar}, that any integral algebraic variety $X$ equipped with a dominant morphism $v  : \Spec(V) \to X$ from a valuation ring $V$ can be ``desingularized along $V$'': there  exists a proper birational map $\tilde{X} \to X$ for which the lift $\tilde{v} : \Spec(V) \to \tilde{X}$ of $v$ supplied by the valuative criterion of properness  factors through the regular locus of $\tilde{X}$. This implies the following theorem.

\begin{Theorem}(Zariski) \label{z} Every valuation ring $V$ containing a field $K$ of characteristic zero is a filtered direct limit of smooth $K$-subalgebras of $V$(in particular they are regular rings).
\end{Theorem}

A smooth algebra is here finitely presented.
A ring map $A \to A'$ is \emph{ind-smooth} if $A'$ is a filtered direct limit of smooth $A$-algebras. A filtered direct limit (in other words a filtered colimit) is a limit indexed by a small category that is filtered (see \cite[002V]{SP} or \cite[04AX]{SP}). A filtered  union is a filtered direct limit in which all objects are subobjects of the final colimit, so that in particular all the transition arrows are monomorphisms. The above theorem says that $K\to V$ is ind-smooth. Actually, Zariski proved  that $V$ is a filtered  union of smooth $K$-subalgebras of $V$.
One goal of this paper is to show a weaker statement:

\begin{Theorem} \label{m}
 Let $V\subset V'$ be an immediate extension of valuation rings containing $\bf Q$.  Then $V\subset V'$ is ind-smooth.
  If  $\dim V=\dim V'=1$ and the residue field extension of $V\subset V'$is trivial then $V\subset V'$ is ind-smooth if   and only if the value group extension of $V\subset V'$ is trivial. 
\end{Theorem}

The proof follows from Theorem \ref{T0}, Lemma \ref{l 5.1} and Proposition \ref{one}.  The above result was stated  by mistake  in \cite{Po1} in a more general case. A different proof of  Theorem \ref{z} is given in Theorem \ref{z1} and the method use some facts from model theory described below.

\bbpp[The method of proof]

To achieve the desingularization claimed in Theorem \ref{m}, we replace the initial $V$ by the limit $\tilde{V}$ of a certain countable tower of iterated ultrapowers of $V$, constructed in such a way that $\tilde{V}$ would, in turn, be an immediate extension of a filtered increasing union of valuation rings for which one knows local uniformization. To then conclude, we argue that large immediate extensions and ultrapowers interact well with desingularization.

The techniques we use   include extensions of steps of the General N\'{e}ron desingularization, notably, Lemma \ref{k} that is also key for reductions to complete rank $1$ cases. In the purely transcendental case, Kaplansky's classification \cite{Kap} of Ostrowski's pseudo-convergent sequences plays an important role. 

The utility of desingularizing immediate extensions is evident already in the case when $V'$ is complete of rank $1$ with a finitely generated value group $\Gamma'$. Such a $V'$ has a coefficient field $k$, so, by choosing a presentation $\Gamma' \cong {\bf Z} \val(x_1) \oplus \cdots \oplus {\bf Z}\val(x_n)$, one obtains the immediate extension $V' \cap k(x_1, \dotsc, x_n) \subset V'$. To show that $V=k\subset V'$ is ind-smooth, it  remains to observe that a local uniformization of $V' \cap k(x_1, \dotsc, x_n)$ may be constructed using Perron's algorithm in the style of Zariski.

The goal of the tower of ultrapowers argument given in the  Appendix is to overcome the obstacle that in general $\Gamma$ may not be finitely generated and there may not even be a group section $s : \Gamma \to K^*$ to $\val:K^*\to \Gamma$ (roughly, such an $s$ suffices).  Nevertheless, $s$ can always be arranged for any finitely generated submonoid of $\Gamma$, and the idea is to then use the following fact from model theory: for a system of equations whose finite subsystems have solutions in $V$, the entire system has a solution in a well-chosen ultrapower of $V$ (see the Appendix).

This fact, which rests on the Keisler--Kunen theorem about the existence of good ultrafilters, permits us to obtain $s$ at the expense of passing to an ultrapower. However, such a passage replaces $\Gamma$ by its corresponding ultrapower $\Gamma^*$ and, in order to extend $s$ to this $\Gamma^*$, we then need another ultrapower and some model-theoretic facts about algebraic compactness that ensure that $\Gamma \to \Gamma^*$ be a split injection. Even though the new ultrapower again enlarges the value group, by repeating the construction  countably many times, in the limit we find our final $s$ and can conclude.

We should mention that the proof of the main part of Theorem \ref{m} uses only Lemma \ref{k} and Corollary \ref{c0}. The last sentence from Theorem \ref{m} needs the method explained above together with some facts from Andr\'e homology. 
\eepp
\bsk

We owe thanks to K\k{e}stutis \v{C}esnavi\v{c}ius especially for the Appendix,  but also for many ideas and his  great help on the presentation of the paper. Also we owe thanks to the referees who pointed out several mistakes in a previous version of the paper especially in the former Proposition \ref{p}.

\vskip 0.5 cm
\section{A reduction to the case of complete valuation rings of rank $1$}

 We begin by reviewing the following class of generators of the singular ideal.

For a finitely presented ring map $A \to B$, an element $b \in B$ is \emph{standard over $A$} if there exists a presentation $B \cong A[X_1, \dotsc, X_m]/I$ and $f_1, \dots, f_r \in I$ with $r \leq m$ such that $b = b'b''$ with $b' = \det ({(\partial f_i/\partial X_{j}))_{1 \le i,\, j \le r}} \in A[X_1, \dotsc, X_m]$ and a $b'' \in A[X_1, \dotsc, X_m]$ that kills~$I/(f_1, \dotsc, f_r)$ (our standard element is a special power of the standard element from  \cite[Definition, page 9]{S} given in the particular case of the valuation rings). 
Any multiple of an element $b $  standard over $A$ is   standard over $A$. 
The definition is compatible with base change: more precisely, for any morphism $A \to A'$,  elements of $B$ standard over $A$ map to  elements of $B \tensor_A A'$ standard over $A'$.

\begin{Lemma} \label{std-lem}
For a finitely presented ring map $A \to B$, the loci of vanishing of standard over $A$ elements of $B$ cut out the locus of non-smoothness of $\Spec(B) \to \Spec(A)$. The radical of the ideal generated by the  elements of $B$ standard over $A$ is $H_{B/A}$. 
\end{Lemma}

\begin{proof}
The argument is standard (compare with \cite{El}, \cite[4.3]{S}) but we include it due to the lack of a convenient reference.

If $b \in B$ is standard over $A$, then $B_b$ is the localization of the standard smooth $A$-algebra $(A[X_1, \dots, X_n]/(f_1, \dots, f_r))_{\det ({(\partial f_i/\partial X_{j}))_{1 \leq i,\, j \leq r}}}$ (see \cite[00T8]{SP}), so is $A$-smooth. Conversely, if $B_b$ is the coordinate ring of a smooth neighborhood of a fixed prime $\pp \subset B$, then we may choose a presentation $B[X_1, \dotsc, X_m]/I$ and  $f_1, \dotsc, f_r \in I$ such that, at the expense of localizing at $\pp$ further. The module $(I/I^2)_b$ is a free $B_b$-module with a basis given by the classes of $f_1, \dotsc, f_r$ and $(I/I^2)_b \to \bigoplus_{i = 1}^m B_b \cdot dX_i$ is a split injection such that $dX_{r + 1}, \dots, dX_m$ maps to a basis for the quotient. The first condition and  Nakayama's lemma \cite[00DV]{SP} then supply an $i \in I$ with $(1 +i/b^n)I_b \subset (f_1, \dots, f_r)_b$ for some $n > 0$. It follows that $b^N(b^n + i)$ for some $N > 0$ kills $I/(f_1, \dotsc, f_n)$ and maps to a power of $b$ in $B$. The second condition implies that $b'= \det {(\partial f_i/\partial X_{j}))_{1 \leq i,\, j \leq r}}$ is a unit in $B_b$, so that $b'$ divides some power of $b$ in $B$. In conclusion, some power of $b$ is standard over $A$, as desired.
\hfill\ \end{proof}

To stress the relevance of the desingularization lemma \ref{k}, we recall the following well-known lemma (see \cite[(1.5)]{S} or \cite[07C3]{SP}) and definitions, which will be crucial throughout this paper.

\begin{Lemma} \label{factor-through}
For a ring $R$ and a set $\mathcal S$ of finitely presented $R$-algebras, an $R$-algebra $R'$ is a filtered direct limit of elements of $\mathcal S$ if and only if every $R$-morphism $B \to R'$ with $B$ a finitely presented $R$-algebra factors as $B \to S \to R'$ for some $S \in \mathcal S$. 
\end{Lemma}

By \cite[1.8]{P} (see also \cite[07GC]{SP}), a map of Noetherian rings is ind-smooth if and only if $A'$ is $A$-flat and has geometrically regular $A$-fibers. In particular, a field extension $K'/K$ is ind-smooth if and only if it is separable.

Concretely, by  Lemma \ref{factor-through}, a ring map $A \to A'$ is ind-smooth if and only if every factorization $A \to B \to A'$ with $B$ finitely presented over $A$ can be refined to $A \to B \to S \to A'$ with $S$ smooth (or merely ind-smooth) over $A$.  Thus, a finite product or a filtered direct limit of ind-smooth $A$-algebras is ind-smooth. Evidently, ind-smooth morphisms are stable under base change. They are also stable under compositions, in fact, we have the following slightly finer criterion.

\begin{Lemma} \label{ind-sm-comp}
For an ind-smooth map $A \to A'$ and a map $A' \to A''$ such that for every factorization $A \to B \to A''$ with $B$ finitely presented over $A$ the induced factorization $A' \to A' \tensor_A B \to A''$ can be refined to $A' \to A' \tensor_A B \to S' \to A''$ for some smooth $A'$-algebra $S'$, the map $A \to A''$ is ind-smooth. In particular, the composition of ind-smooth maps is ind-smooth.
\end{Lemma}

\begin{proof}
It suffices to argue that the map $A \to A' \to S'$ is ind-smooth. For this, we express $A'$ as a filtered direct limit  of smooth $A$-algebras $S_i$, note that $S'$ descends to a smooth $S_i$-algebra $S'_i$ for some $i$, and conclude that $S'$ is then the filtered direct limit of the smooth $A$-algebras $S'_j = S_j \tensor_{S_i} S'_i$ with $j \geq i$.
\hfill\ \end{proof}

The following lemma originates in \cite[(7.1)]{P0} and its variants have appeared, for instance, in \cite[18.1]{S}, \cite[7.2]{P}, \cite[07CT]{SP}, \cite[Proposition 3]{ZKPP}, and \cite[Proposition 5]{P1}. The version below differs in two aspects: we do not assume Noetherianness and do not require the elements $a$ or $b$ to come from the base ring $A$. The latter improvement is particularly convenient for our purposes---we recall that in the General N\'{e}ron desingularization arranging for $b$ to come from $A$ is an additional step before one can apply the desingularization lemma (compare with, for instance, \cite[07F4]{SP}).

\begin{Lemma} \label{k}
For a commutative diagram of ring morphisms

\xymatrix@R=0pt{
& B \ar[rd] & & && B \ar[dd]^{b \, \mapsto\, a} \ar[rd] & \\
A \ar[rd] \ar[ru] & & V & \mbox{that factors as follows} & A \ar[ru]\ar[rd] & & V/a^3V \\ 
& A' \ar[ru] & & && A'/a^3A' \ar[ru] &
}

\noindent with $B$ finitely presented over $A$, a $b \in B$ that is standard over $A$, and a nonzerodivisor $a \in A'$ that maps to a nonzerodivisor in $V$ that lies in every maximal ideal of $V$, 
there is a smooth $A'$-algebra $S$ such that the original diagram factors as follows:

\hskip 4 cm\xymatrix@R=0pt{
& B \ar[rdd] \ar[rrd] & & \\
A \ar[rd] \ar[ru] & &  & V. \\
& A' \ar[r]  & S \ar[ru] &
}

 \end{Lemma}

\begin{proof}
The finitely presented $A'$-algebra $B \tensor_A A'$ comes equipped with a morphism to $V$ and a retraction modulo $b^3$ to $A'/a^3A'$ that sends $b$ to $a$. Moreover, the image of $b$ in $B \tensor_A A'$ is standard over $A'$. Thus, by replacing $A$ by $A'$ and $B$ by $B \tensor_A A'$, we reduce to the case $A = A'$. 

Since the images of $a$ and $b$ in $V$ agree modulo $a^3V$, these images are unit multiples of each other. We  write 
\[
B = A[X_1, \dots, X_m]/I \ \  \mbox{and} \ \ f_1, \dots, f_r \in I 
\]
and choose $b' = \det ({(\partial f_i/\partial X_{j})_{1 \leq i,\, j \leq r}}) \in A[X_1, \dotsc, X_m]$ and a $b'' \in  A[X_1, \dotsc, X_m]$ that kills $I/(f_1, \dotsc, f_r)$ with $b = b'b''$ in $B$. In these coordinates, we fix a map

\xymatrix@R=0pt@C=10pt{
&& & & A[X_1, \dotsc, X_m] \ar[dd] \ar@{->>}[r] & B \ar@{->>}[r] & B/b^3B \ar[dd] \\ 
A[X_1, \dotsc, X_m] \ar[rr]^-{f\mapsto \tilde{f}} && A & \mbox{that makes the diagram} &
\\
&&&& A \ar@{->>}[rr] && A/a^3A
}
\noindent commute, so that $\tilde{f}\in a^3A$ for every $f \in I$. In particular, the assumption $b \mapsto a$  gives $\widetilde{b'b''} \equiv a \bmod a^3A$. It follows that 
\[
\widetilde{b'b''} = a u \ \ \mbox{for some} \ \ u\in 1+a^2A,
\]
\noindent so that, in particular, $u$ maps to a unit in $V$.

We consider the $m \times m$ matrix $\Delta$ given by
\[
 \left({\begin{matrix} 
 \partial f_1/\partial X_1 & \partial f_1/\partial X_2 & \dotsc & \partial f_1/\partial X_r & \partial f_1/\partial X_{r + 1} & \partial f_1/\partial X_{r + 2} & \dotsc & \partial f_1/\partial X_m \\
\partial f_2/\partial X_1 & \partial f_2/\partial X_2 & \dotsc & \partial f_2/\partial X_r & \partial f_2/\partial X_{r + 1} & \partial f_2/\partial X_{r + 2} & \dotsc & \partial f_2/\partial X_m \\
\vdots & \vdots  & \ddots & \vdots & \vdots & \vdots & \ddots & \vdots \\
\partial f_r/\partial X_1 & \partial f_r/\partial X_2 & \dotsc & \partial f_r/\partial X_r & \partial f_{r}/\partial X_{r + 1} & \partial f_r/\partial X_{r + 2} & \dotsc & \partial f_r/\partial X_m \\
0 & 0 & \dotsc & 0 & 1 & 0 & \dotsc & 0 \\
0 & 0 & \dotsc & 0 & 0 & 1 & \dotsc & 0 \\
\vdots & \vdots  & \ddots & \vdots & \vdots & \vdots & \ddots & \vdots \\
0 & 0 & \dotsc & 0 & 0 & 0 & \dotsc & 1 \\
\end{matrix} } \right)
\]
that satisfies $\det(\Delta) = b'$. We let $Ad(\Delta)$ denote the adjoint matrix, so that 
\[
Ad(\Delta) \cdot \Delta = \Delta \cdot Ad(\Delta) = b' \cdot Id_{m \times m}.
\] 
 We let $x_i$ and $x_i'$ be the images in $V$ of $X_i$ and $\tilde{X}_i$, respectively, so that, by construction, $x_i - x'_i \in a^3 V$. Moreover, $a$ is a nonzerodivisor in $V$ and there we have that
 \[
\left({ \begin{smallmatrix} t_1 \\ \vdots \\ t_m \end{smallmatrix}}\right)    \tilde{\Delta} \ \  \left({\begin{smallmatrix} (x_1 - x_1')/a^2 \\ \vdots \\ (x_m - x_m')/a^2 \end{smallmatrix}}\right)  
\ \ \mbox{satisfies} \ \ t_i \in aV \ \ \mbox{and} \ \ 
a\widetilde{b''Ad(\Delta)} \left({ \begin{smallmatrix} t_1 \\ \vdots \\ t_m \end{smallmatrix}}\right) = u  \left({\begin{smallmatrix} x_1 - x_1' \\ \vdots \\ x_m - x_m' \end{smallmatrix}}\right).
 \]
We let $T_1,\ldots,T_m$ be new variables and set
\[
\left({ \begin{smallmatrix} h_1 \\ \vdots \\ h_m \end{smallmatrix}}\right) = u\left({ \begin{smallmatrix} X_1- \tilde{X}_1 \\ \vdots \\ X_m- \tilde{X}_m \end{smallmatrix}}\right)-a\widetilde{b''Ad(\Delta)} \left({ \begin{smallmatrix} T_1 \\ \vdots \\ T_m \end{smallmatrix}}\right), \ \ \mbox{so that} \ \ h_i \in A[X_1, \dots, X_m, T_1, \dots, T_m].
\]
By construction, if we map $T_i$ to $t_i$ in $V$, then the $h_i$ map to $0$, so we obtain the map
 \[
\phi: A_u[X_1, \dots, X_m, T_1, \dots, T_m]/(h_1, \dots, h_m) \to V \ \ \mbox{given by} \ \ X_i \mapsto x_i, \ \ T_i \mapsto t_i.
 \] 
Since we have inverted $u$, the source of this map may be identified with $A_u[T_1, \dotsc, T_m]$. To proceed further, we will use Taylor's formula to express each $f_i$ in terms of this identification.

  By Taylor's formula, for any ring $R$, any section $R[X_1, \dotsc, X_m] \xrightarrow{f\, \mapsto\, \tilde{f}} R$, and any $f \in R[X_1, \dotsc, X_m]$,
 \[
  f - \tilde{f} - \sum_{i = 1}^m \widetilde{(\partial f/\partial X_i)} (X_i - \tilde{X_i}) \in (X_1 - \tilde{X_1}, \dots, X_m - \widetilde{X_m})^2 \subset R[X_1, \dots, X_m].
  \]
In particular, by applying this with $R = A[T_1, \dots, T_m]$ and letting $d$ denote the maximal total degree of any monomial that appears in some $f_i$, we obtain 
\[
Q_i\in (T_1, \dotsc, T_m)^2 \subset A[T_1, \dots, T_m]
\]
for which 
\[ \begin{aligned} 
&u^df_i-u^d{\tilde f}_i\equiv u^{d - 1}a\tilde{b''} \ \ ( \widetilde{(\partial f_i/\partial X_1)}, \dots,\widetilde{(\partial f_i/\partial X_m)})
 \widetilde{Ad(\Delta)} \left({ \begin{smallmatrix} T_1 \\ \vdots \\ T_m \end{smallmatrix}}\right)+\\  
 &a^2  Q_i \bmod (h_1, \dots, h_m)  \equiv u^{d - 1}a\tilde{b''}\tilde{b'}T_i + a^2  Q_i \equiv a^2u^{d} T_i + a^2  Q_i \bmod (h_1, \dots, h_m)
\end{aligned}\]
We have $\tilde{f_i} = a^2 b_i$ for some $b_i\in aA$, and for $1 \leq i \leq r$ we set
\[
g_i =  u^db_i+u^d T_i+Q_i\in A[T_1, \dots, T_m], \ \ \mbox{so that} \ \ a^2g_i \equiv u^d f_i  \bmod (h_1, \dotsc, h_m).
\]
This achieves the promised expression of $f_i$ in terms of the identification of the source of $\phi$ with $A[T_1, \dotsc, T_m$] and simultaneously shows that each $g_i$ vanishes in $V$, so that $\phi $ induces a map 
\[
\phi \colon A_u[X_1, \dots, X_m, T_1,\ldots,T_m]/(I, g_1, \dots, g_r,  h_1, \dots, h_m) \to V.
\] 
In $A[X_1, \dots, X_m]$ the element $b'b'' - au$ lies in the ideal $(X_1 - \tilde{X}_1, \dots, X_m  - \tilde{X}_m)$, so in the quotient $A_u[X_1, \dots, X_m, T_1, \dots, T_m]/(h_1, \dots, h_m)$ it lies in the ideal $a(T_1, \dots, T_m)$. It then follows from the definition of $b''$ and the fact that after inverting $u$ and modulo $(h_1, \dots, h_m)$ the ideal $(g_1, \dots, g_r)$ contains $(f_1, \dots, f_r)$ that some element from the coset $a(u + (T_1, \dots, T_m))$ kills the image of $I$ in
\[
A_u[X_1, \dots, X_m, T_1, \dots, T_m]/(g_1, \dots, g_r, h_1, \dots, h_m).
\]
Setting $u' = \det ((\partial g_i/\partial T_j)_{1\leq i, j \leq r})$, we deduce that the same then holds in the localization 
\[\begin{aligned}
(A_u[X_1, \dots, X_m, T_1, \dots, T_m]/&(g_1, \dots, g_r, h_1, \dots, h_m))_{u'} \cong \\ &(A_u[T_1, \dots, T_m]/(g_1, \dots, g_r))_{u'}.
\end{aligned}\]
However,  the latter is smooth over $A$, to the effect that $a$ is a nonzerodivisor in the ring above. It follows that even some element $u'' \in u + (T_1, \dots, T_m)$ kills the image of $I$ in the ring above. By construction, both $u'$ and $u''$ map to units in $V$   and $\phi$ factors through the $A$-smooth algebra
\[
S = (A_u[X_1, \dots, X_m, T_1, \dots, T_m]/(g_1, \dots, g_r, h_1, \dots, h_m))_{u'u''}. 
\]
\hfill\ \end{proof}

In some situations, when applying Lemma \ref{k} we will not initially have a map $A' \to V$. The following lifting lemma will help to bypass this obstacle. Its key novel aspect is that the elements $s$, $s'$, and $v$ need not come from the base ring $A$ (compare with \cite[(8.1]{P0}, \cite[(17.1)]{S}, or \cite[07CP]{SP}). 

\begin{Lemma} \label{lift-lem}
For a ring morphism $A \to V$ with $V$ local, a smooth $A$-algebra $S$, an element $s \in S$, a nonunit $v \in V$, and a factorization
\[
A \to S \xrightarrow{s\, \mapsto\, v} V/v^{n}V \ \ \mbox{for some} \ \ n \ge 2,
\]
there are a smooth $A$-algebra $S'$, an element $s' \in S'$, and factorizations
\[
A \to S' \xrightarrow{s' \, \mapsto\, uv} V \ \ \mbox{with} \ \ u \in V^* \ \ \mbox{and} \ \ A \to S \xrightarrow{s\, \mapsto\, s'} S'/s'^nS' \to V/v^nV;
\]
if $s$ is the image of an element $a \in A$, then one may choose $s' = a$.
\end{Lemma}

\begin{proof}
Due to the local structure of smooth and \'{e}tale morphisms \cite[054L,00UE]{SP}, by localizing $S$ around the preimage of the maximal ideal of $V$, we may assume that $S$ is standard \'{e}tale over a polynomial $A$-algebra, that is, that
\[
 S \cong (A[X_1, \dots, X_d, Y]/(f))_{g \cdot {\partial f}/{\partial Y}} \ \ \mbox{for some} \ \ f, g \in A[X_1, \dots, X_d, Y] \ \ \mbox{with $f$ monic in $Y$.}
\]
For a suitable $n\in {\bf N}$, some unit multiple of $s \in S$ of the form  $(g \cdot ( {\partial f}/{\partial Y}))^N\cdot s$  lifts to an $\tilde{s} \in A[X_1, \dots, X_d, Y]$. Letting $x_1, \dots, x_d, y$ be some lifts to $V$ of the images of $X_1, \dots, X_d, Y$ in $V/v^nV$, we find that the $A$-morphism

$
A[X_1, \dotsc, X_d, Y] \xrightarrow{X_i\, \mapsto\, x_i,\, Y\, \mapsto\, y} V 
$

\noindent maps $\tilde{s}$ to a unit multiple of $v$ (as may be checked modulo $v^{n}$), so it maps $f$ to $\tilde{s}^{n} w$ for some $w \in V$. Thus, we obtain the $A$-morphism
\[
S' = (A[X_1, \dots, X_d, Y, W])_{g \cdot {\partial f}/{\partial Y} \cdot {\partial (f - \tilde{s}^n W)}/{\partial Y}}/(f - \tilde{s}^n W) \xrightarrow{W\, \mapsto\, w} V.
\]
By construction, $S'$ is $A$-smooth and, setting $s' = \tilde{s} \cdot (g \cdot {\partial f}/{\partial Y})^{-N}$ in $S'$ we have the identification
\[
S'/s'^nS' \cong (A[X_1, \dots, X_d, Y, W])_{g \cdot {\partial f}/{\partial Y} \cdot {\partial (f - \tilde{s}^n W)}/{\partial Y}}/(f, s'^n) \cong ((S/s^nS)[W])_{{\partial (f - \tilde{s}^n W)}/{\partial Y}}
\]
with $s'$ corresponding to $s$ and compatibly with the maps to $V/v^nV$. The main part of the claim follows, and for the remaining assertion about $a$ note that if $s$ is the image of an $a \in A$, then we may choose $N = 0$ and $\tilde{s} = s' = a$ above.
\hfill\ \end{proof}

For desingularizing valuation rings, the above lemmas will be useful in several different ways. We illustrate this right away with the following results that facilitate passage to completions.

\begin{Proposition} \label{pass-to-hat}
For a  ring $A$, a dense extension of valuation rings (see Section 3) $V\subset V'$, $K$ the  fraction field of $V$,  a ring morphism $A \to V$, a finitely presented $A$-algebra $B$, and maps
\[
A \to B \to V \ \ \mbox{such that} \ \ B \to K \ \ \mbox{factors through some $A$-smooth localization of $B$}
\] 
suppose that
 there exist a smooth $A$-algebra $S'$ and a factorization $
A \to B \to S' \to V'$. Then
 there exist a smooth $A$-algebra $S$ and a factorization $
A \to B \to S \to V$. 
In particular, there exist a smooth $A$-algebra $S$ and a factorization $A \to B \to S \to V$   if  there exist a smooth $A$-algebra $\hat S$ and a factorization $A \to B \to {\hat S} \to {\hat V}$,  ${\hat V}$ being the completion of $V$. 
\end{Proposition}

\begin{proof} By hypothesis $H_{B/A}V\not =0$ and let $b\in H_{B/A}V$, $b\not =0$. Let  $B \cong
A[Y]/I$, $Y =
(Y_1\ldots, Y_m)$, $I$ being  a finitely generated ideal. Changing $A$ by $A[Z]$, $B$ by $B[Z]$, the map $B[Z]\to V$ being given by $Z\to b$, we may assume that $b$ comes in fact from $A$. Indeed, if $S$ is given for $B$, let us say as in (1) then $S[Z]$ could be taken for $B[Z]$ as in (1). Similarly as in  \cite[Lemma 4]{ZKPP} we may
assume that   for some polynomials $ f = (f_1,\ldots,f_r)$
from $I$, we have  $b\in NM B$ for some $N\in ((f):I)$ and a $r\times r$-minor $M$ of the Jacobian matrix $(\partial f_i/\partial Y_j)$. Thus we may assume $b$ is standard for $B$ over $V$, which is necessary later to apply Lemma \ref{k}.
 Note that the composite map $B\to V\to V/b^3V\cong V'/b^3V'$ factors through a smooth $A/b^3A$-algebra.  By Lemma \ref{k} $B\to V$  factors through a smooth $A$-algebra as well.
 
 Indeed,   
 since $V/b^3V \cong V'/b^3V'$, Lemma \ref{lift-lem} supplies a smooth $A$-algebra $S'_0$, an $s' \in S'$, a factorization $A \to S'_0 \to V$ that sends $s'$ to a unit multiple of $b$ in $V$, and a factorization
\[
\xymatrix@C=10pt@R=0pt{
     & B/b^3B \ar[r] & S'/b^3S' \ar[rd] \ar[dd]_-{b\, \mapsto\, s'} & \\
      A \ar[ru]\ar[rrd] & && V/b^3V. \\ 
    &  &   S'_0/s'^3S'_0 \ar[ru] &
}
\]
The local ring of $S_0'$ at the preimage of the maximal ideal of $V$ is a domain (see \cite[033C]{SP}, ) and $s'$ is nonzero in this local ring, so it is a nonzerodivisor there. Thus, Lemma \ref{k} applies and supplies a smooth $S'_0$-algebra $S''$ with a factorization $A \to B \to S'' \to V$. Note that $S''$ is a  smooth $A$-algebra.
\hfill\ \end{proof}

To draw further consequences, we will use the following well-known result of Nagata (see \cite[Theorem 4]{Na} or \cite[053E]{SP}).

\begin{Lemma} \label{fp-lemma}
Any finitely generated, flat (equivalently, torsion free) algebra over a valuation ring is finitely presented. 
\end{Lemma}

\begin{Corollary} \label{descend-from-hat}
For a local injection $V \to V'$ of valuation rings that induces a separable extension $K'/K$ of fraction fields, if the map $V \to \tilde{V'}$ is ind-smooth,  ${\tilde V}'$ being the completion of $V'$, then so is $V \to V'$.
\end{Corollary}

\begin{proof}
The separability assumption and Lemma \ref{fp-lemma} imply that Proposition \ref{pass-to-hat} applies to every finite type $V$-subalgebra $B \subset V'$: a limit argument reduces to showing that the smooth locus of $B$ over $V$ is nonempty, which follows from the separability of $\Frac(B)/K$ thanks to \cite[(6.7.4.1) in IV2]{EGA} and \cite[(17.5.1) in IV4]{EGA}. It then remains to apply Lemma \ref{factor-through}. 
\hfill\ \end{proof}

The work above allows us to relate certain ``formal desingularization'' extensions of valuation rings studied in \cite[section 6]{Po2} to ``weak desingularization'' (that is, ind-smooth) extensions as follows.

\begin{Proposition} \label{fd-ext}
Fix a local injection $V \to V'$ of valuation rings with fraction fields $K \to K'$ such that $\val(V\setminus \{ 0\})$ is cofinal in $\val(V'\setminus \{ 0\})$ and for each $0 \neq v \in V$ the map $V/vV \to V'/vV'$ is ind-smooth, a finitely presented $A$-algebra $B$, and maps $V \to B \to V'$ such that the map $B \to K'$ factors through some $V$-smooth localization of $B$. There is a smooth $V$-algebra $S$ and a factorization $V \to B \to S \to V'$.  If, in addition, $K'/K$ is separable, then $V'$ is ind-smooth over $V$.
\end{Proposition}

\begin{proof}
In the case when $K'/K$ is separable, $B$ could be any finite type $V$-subalgebra of $V'$, which is finitely presented by Lemma \ref{fp-lemma}. So the last assertion follows from the rest and Lemma \ref{factor-through}. For the assertion about $B$, we use Lemma \ref{std-lem} to choose an  element $b \in B$ standard over $V$  that does not die in $V'$. We assume that $b$ is not a unit in $V'$ (or else we may set $S = B_b$) and we choose a $0 \neq v \in V$ with $\val(v) > \val(b)$,
where $\val(b)$ denotes the valuation of $b$ considered in $V'$. By our assumptions, there are a smooth $V/v^3V$-algebra $\bar{S}$, an $s \in \bar{S}$, and a factorization $V \to B \to \bar{S} \xrightarrow{s\, \mapsto\, v} V'/v^3V'$ such that $b \mid s$ in $\bar{S}$. Thus, since $\bar{S}$ lifts to a smooth $V$-algebra (see \cite[(1.3.1)]{Ara} or  \cite[07M8]{SP}), Lemma \ref{lift-lem} supplies a smooth $V$-algebra $S'$, an $s' \in S'$, a factorization $V \to S' \to V'$ that sends $s'$ to a unit multiple of $v$, and a factorization 
$
V \to B \to \bar{S} \xrightarrow{s\, \mapsto\, s'} S'/s'^3S' \to V'/v^3V'.
$
Since $b\mid s'$ in $S'/s'^3S'$, by replacing $S'$ by its localization by an element of $1 + s'^2S'$ if necessary we may ensure that $a \mid s'$ in $S'$ for some lift $a \in S'$ of $b \in S'/s'^3S'$. Then we  have a factorization
\[
V \to B \xrightarrow{b\, \mapsto\, a} S'/a^3S' \to V'/a^3V'.
\]
As in the proof of Lemma \ref{pass-to-hat}, Lemma \ref{k} then supplies a  smooth $V$-algebra $S$. 
\hfill\ \end{proof}

The following localization lemma, a variant of \cite[Lemma 2]{P'}, \cite[(12.2)]{S}, or \cite[07F9]{SP}, will permit us to localize our valuation rings when arguing their ind-smoothness. 

\begin{Lemma} \label{delocalize}
For ring maps $A \to B \to V$ with $B$ of finite type over $A$, a prime $\Pp \subset V$ with preimage $\pp \subset A$, and a factorization $A \to B \to S' \to V_{\Pp}$ for a finitely presented $A_{\pp}$-algebra $S'$, there are a finitely presented $A$-algebra $S$, an $s \in S$ with $S_{s} \tensor_A A_{\pp} \simeq S'[X, X^{-1}]$, and a  factorization
\[
 A \to B \to S \to V \ \ \mbox{such that} \ \ S \to V_{\Pp} \ \ \mbox{factors as} \ \ S \to S_{s} \tensor_A A_{\pp} \to  V_{\Pp}.
\] 
\end{Lemma}

\begin{proof}
Following the argument of \cite[(12.2)]{S}, we choose a presentation
\[
S' \simeq (B \tensor_A A_{\pp})[X_1, \dots, X_n]/(f_1(X_1, \dots, X_n), \dots, f_m(X_1, \dots, X_n))
\]
(see \cite[00F4]{SP}) in which the polynomials $f_i$ have coefficients in $B$, and we set
\[
S = B[X_0, X_1, \dots, X_n]/(X_0^Nf_1(X_1/X_0, \dots, X_n/X_0), \dots, X_0^Nf_m(X_1/X_0, \dots, X_n/X_0))
\]
for a large enough $N > 0$ for which each $X_0^Nf_i(X_1/X_0, \dots, X_n/X_0)$ is a (necessarily homogeneous) polynomial in $X_0, X_1, \dotsc, X_n$ of positive degree and coefficients in $B$. We set $s = X_0$, so that a desired isomorphism $S_{s} \tensor_A A_{\pp} \simeq S'[X, X^{-1}]$ is induced by the change of variables $X_0 \mapsto X$ and $X_i \mapsto XX_i$ for $1 \leq i \leq n$. To build the map $S \to V$, we first choose $x_1, \dots, x_n \in V$ and $t \in V \setminus \Pp$ such that $X_i$ maps to $x_i/t \in V_{\Pp}$. Continuing to use abusive notation for homogeneous polynomials, we note that the (``homogeneous'' in $t, x_1, \dots, x_n$) elements $t^Nf_i(x_1/t, \dots, x_n/t)$ of $V$ die in $V_{\Pp}$, so they are killed by some $t' \in V \setminus \Pp$. Thus, the $B$-morphism
\[
B[X_0, X_1, \dots, X_n] \to V \ \ \mbox{given by} \ \ X_0 \mapsto t't, X_1 \mapsto t'x_1, \dots, X_n \mapsto t' x_n
\]
factors through $S$. By construction, the resulting morphism $S \to V_{\Pp}$ factors through the localization $S_{s} \tensor_A A_{\pp}$ of $S$, as desired.
\hfill\ \end{proof}

We are ready for the promised reduction to complete, one-dimensional valuation rings.

\begin{Proposition} \label{reduce-to-ht-1}
Consider the following property of a valuation ring $V$ and a subring $A \subset V$:

$(*)$ every $A \to B \to V$ with $B$ a finite type $A$-algebra such that $B \to \Frac(V)$ factors through  
 an $A$-smooth localization of $B$ has a refinement $A \to B \to S \to V$ with $S$ smooth over $A$.

For a finite dimensional valuation ring $V$ with a subfield $A\subset V$, if for all consecutive primes $\qq' \subset \qq \subset V$ the complete height one valuation ring $\widetilde{(V/\qq')_{\qq}}$ satisfies $(*)$, then so does $V$.
\end{Proposition}

\begin{proof}
We fix a finite type $A$-algebra $B$ equipped with a factorization $A \to B \to V$ as in  $(*)$, which we need to factor further as $A \to B \to S \to V$ for some smooth $A$-algebra $S$.  When $B \to V$ itself factors through an $A$-smooth localization of $B$, there is nothing to show. Otherwise, since $V$ is of finite height, we may choose the minimal prime $\qq \subset V$ whose preimage in $B$ does not lie in the $A$-smooth locus of $\Spec(B)$ and the largest prime $\qq' \subsetneq \qq \subset V$ properly contained in $\qq$ (the assumption in $(*)$ ensures that $\qq'$ exists). Thanks to Lemma \ref{delocalize}, we may replace $V$ by $V_{\qq}$ to reduce to the case when $\qq$ is the maximal ideal (so that $\widetilde{(V/\qq')_{\qq}} = \widetilde{V/\qq'}$): indeed, once we resolve this case, then, by using Lemma \ref{delocalize}, we will be able to refine $B$ to an $A$-algebra that either is smooth or for which $\qq$ is strictly larger, and, by iteration, we will then arrive at a desired $S$.

By Lemma \ref{std-lem}, there is an element $b \in B$ standard over $A$  that maps to $\qq \setminus \qq'$. The property from $(*)$ of $\widetilde{V/\qq'}$ then supplies a smooth $A$-algebra $S'$, an element $s \in S'$ (the image of $b$), and a factorization 
\[
\ \ \xymatrix@R=0pt{
  & B \ar[dd]^{b \, \mapsto\, s} \ar[rd] & \\
   A \ar[ru]\ar[rd] & & \widetilde{V/\qq'}/b^3\widetilde{V/\qq'} \cong V/b^3V. \\ 
 & S'/s^3S' \ar[ru] &
}
\]
Thanks to Lemma \ref{lift-lem}, we may change $S'$ in order to make sure that the map $S'/s^3S' \to V/b^3V$ lifts to an $A$-morphism $S' \to V$. This puts us in a situation in which we may apply Lemma \ref{k} to obtain a smooth $S'$-algebra $S$ with a desired factorization $A \to B \to S \to V$.
\hfill\ \end{proof}

\vskip 0.5 cm

\section{Ind-smoothness of large immediate extensions of valuation rings} \label{large-immediate}

Our next goal is to find a large class of extensions of valuation rings that are ind-smooth. The argument combines classical results from valuation theory that go back to Kaplansky, results from \cite{Po1} (see Lemma \ref{imm-ind-smooth} and its proof), and the desingularization lemmas from Section 2.

Consider the case when $V$ is not noetherian and its associated valuation has rank one.  In the Noetherian case a immediate  extension of valuation rings $V\subset V'$  is  dense,  but in general case it need not be.  If $V\supset {\bf Q}$  the problem is solved  by Ostrowski's Defektsatz  \cite{O} but when the characteristic of the residue field of $V$ is $>0$ the  immediate algebraic extensions present extra difficulties.

An inclusion $V \subset V'$ of valuation rings is an \emph{immediate extension} if it is local as a map of local rings and induces isomorphisms between the value groups and the residue fields of $V$ and $V'$. For such a $V \subset V'$, letting $K'/K$ be the induced fraction field extension, we have $V = V' \cap K$ (see \cite[(4.1) in VI]{Bou}). Moreover, for any subextension $K'/K''/K$ and the valuation ring $V'' = V' \cap K''$, both $V \subset V''$ and $V'' \subset V'$ are then also immediate extensions (to check the value group requirement one uses that any $v'' \in V''$ is a unit if and only if so is its image in $V'$). 

For example, for any valuation ring $V$, the extension $V \subset \tilde{V}$ is immediate (see \cite{S}),  ${\tilde V}'$ being the completion of $V'$.

For a valuation ring $V$ with the fraction field $K$, a sequence $\{v_i \}_{i < \omega}$ in $K$ indexed by the ordinals $i$ less than a fixed limit ordinal $\omega$ is \emph{pseudo-convergent} if 

$\val(v_{i} - v_{i''} ) < \val(v_{i'} - v_{i''} )    \ \ \mbox{(that is,} \ \ \val(v_{i} - v_{i'}) < \val(v_{i'} - v_{i''})  ) \ \ \mbox{for} \ \ i < i' < i'' < \omega$
(see \cite{Kap}, \cite{S}).
A (possibly nonunique) \emph{pseudo-limit} of a pseudo-convergent sequence $\{v_i \}_{i < \omega}$ is an element $\alpha \in K$ with 

$ \val(\alpha - v_{i}) < \val(\alpha - v_{i'}) \ \ \mbox{(that is,} \ \ \val(\alpha -  v_{i}) = \val(v_{i} - v_{i'})) \ \ \mbox{for} \ \ i < i' < \omega$.
A pseudo-convergent sequence $\{v_i\}_{i < \omega}$ in $K$ is 
\begin{enumerate}
\item
\emph{algebraic} if some $f \in K[T]$ satisfies $\val(f(v_{i})) < \val(f(v_{i'}))$ for large enough $ i < i' < \omega$;

\item
\emph{transcendental} if each $f \in K[T]$ satisfies $\val(f(v_{i})) = \val(f(v_{i'}))$ for large enough $i < i' < \omega$.
\end{enumerate}
(Here ``large enough'' means larger than a fixed ordinal $\omega' < \omega$ that is allowed to depend on $f$.)
In both cases, \cite[Theorems 1, 2]{Kap} describe the valuation of $K'$ that extends $V$ of $K$. For instance, in the transcendental case, by \cite[Theorem 2]{Kap}, this valuation on $K(t)$ is given by setting 
\[
\val(((f(t))/(g(t))) = \val(f(v_i)) - \val(g(v_i)) \qxq{for large enough}  i < \omega.
\]
These results lead to \cite[Theorem 4]{Kap}: a valuation ring $V$ has no nontrivial immediate extensions if and only if each pseudo-convergent sequence in its fraction field $K$ has a pseudo-limit in $K$.
If for all $\gamma\in \Gamma$, the value group of $V$, there exists $i<i'$ sufficiently  large such that $\val(v_i-v_{i'})>\gamma$ then we call  
 $\{v_i \}_{i < \omega}$ {\em fundamental}.
 As in \cite[Lemma 3.2]{Po1} we get the following lemma.

\begin{Lemma} \label{imm-ind-smooth}
For an immediate extension $V \subset V'$ of valuation rings and a transcendental pseudo-convergent sequence $(v_i)_{i<\omega}$ in $K$, which has a pseudo-limit $v'$ in $K'$ but no pseudo-limit in $K$ the valuation ring $V''=V'\cap K(v')$  is a filtered  union of smooth $V$-subalgebras. 
\end{Lemma}

\begin{proof}  For each $i$ set
$x_i=(v'-v_i)/(v_{i+1}-v_i),$
 so that $x_i$ is a unit in $V'$.
Let $\mm'$ be the maximal ideal of $V'$. We show that for every  polynomial $0 \neq f \in V[t]$ it holds
\[
f(v') \in f(v_i) \cdot (1 + \mm' \cap V[x_i]) \qxq{for every large enough}  i < \omega.
\]
Since $\{v_i\}_{i < \omega}$ is transcendental, for each $g(t) \in K[t]$, the value $\val(g(v_i))$ is constant for large $i$. Moreover,  for large $i$ the values $\val(v' - v_i)$  are strictly increasing as $i$ increases. Thus, in the Taylor expansion\footnote{The polynomials $D^{(n)}f \in R[t]$ for $f \in R[t]$ make sense for any ring $R$: indeed, one constructs the Taylor expansion in the universal case $R = {\bf Z}[a_0, \dots, a_{\deg f}]$ by using the equality $n! \cdot (D^{(n)} f) = f^{(n)}$.}
\[
\tst  f(v') = \sum_{n = 0}^{\deg f} (D^{(n)}f)(v_i) \cdot (v' - v_i)^n \qxq{with}  D^{(n)}f \in V[t]
\]
the values $\val((D^{(n)}f)(v_i) \cdot (v' - v_i)^n)$ are pairwise distinct for every large enough $i$. Consequently, since $\val(f(v')) = \val(f(v_i))$ for large $i$, we conclude that 
\[
\val((D^{(n)}f)(v_i) \cdot (v' - v_i)^n) > \val(f(v_i)) \qx{for every $n > 0$ and large enough $i < \omega$.}
\]
It  remains to note that 
\[
(v' - v_i)^n = x_i^n \cdot (v_{i + 1} - v_i)^n \ \ \mbox{and} \ \ \val(v' - v_i) = \val(v_{i + 1} - v_i), 
\] 
which is enough for our claim.
In particular, we get that $f(v')$ is transcendental over $K$.
The element $x_i$ is transcendental over $K$, so $V[x_i] \subset V'$ is the polynomial algebra. Moreover, for $i < i' < \omega$ we have $x_{i} = x_{i'} \cdot (v_{i' + 1} - v_{i'})/(v_{i + 1} - v_{i}) + (v_{i'} - v_{i})/ (v_{i + 1} - v_{i})$, so $V[x_{i}] \subset V[x_{i'}] \subset V'$. Consequently,  we arrive at a nested sequence
\[
\{ V[x_{i}]_{\mm' \cap V[x_{i}]} \}_{i < \omega}  \ \ \mbox{of ind-smooth $V$-subalgebras of $V'$,}
\]
and,  it remains to show that every element of $V''$ belongs to some $V[x_{i}]_{\mm' \cap V[x_{i}]}$. In fact,  it suffices to show that each $0 \neq f \in V[t]$ satisfies
\[
f(v') \in f(v_i) \cdot (1 + \mm' \cap V[x_i]) \qxq{for every large enough} i < \omega
\]
which was done above.
\hfill\ \end{proof}

\begin{Lemma} \label{immalg-ind-smooth}
For an immediate extension $V \subset V'$ of valuation rings containing $\bf Q$ with value group $\Gamma\subset {\bf R}$ every  algebraic pseudo-convergent sequence $(v_i)_{i<\alpha}$ in $K$, which is not a fundamental sequence but has a pseudo-limit $v'$ in $K'$ has also a pseudo-limit in $K$. 
\end{Lemma}

\begin{proof}  By \cite[Theorem 3]{Kap} there exists an immediate extension of valued fields $K\subset K(u)$ such that $u$ is algebraic over $K$ and it is a pseudo-limit of $(v_i)$ in $K(u)$. As a consequence of Ostrowski's Defektsatz \cite[Sect 9, No 55]{O} (see  \cite[Corollary 4.2]{Po1}, or \cite[Corollary 3.10]{Po2}) we see that $K\subset K(u)$ is dense, that is, $u$ belongs to the completion of $K$. Thus $(v_i)$ has a pseudo-limit in $K$ by \cite[Lemma 2.5]{Po2}.
\hfill\ \end{proof}

\begin{Remark}{\em The above lemma is false if the characteristic of the residue field of a  valuation rings is $>0$ (see \cite[Example 3.13]{Po2} inspired by \cite[Sect 9, No 57]{O}).}
\end{Remark}

\begin{Proposition} \label{p0} For an immediate extension $V \subset V'$ of valuation rings containing $\bf Q$ with value group $\Gamma\subset {\bf R}$, $V'$ is ind-smooth over $V$.
\end{Proposition}

\begin{proof} Applying Lemma \ref{imm-ind-smooth} possible infinitely, even uncountably many, we find a pure transcendental extension $K''\subset K'$ of $K$ such that $V''=V'\cap K''$ is ind-smooth over $V$ and all transcendental pseudo-convergent sequences in $V''$ over $V$ having a pseudo limit in $V'$, which are not fundamental sequences,  have  pseudo-limits in $V''$. By Lemma \ref{immalg-ind-smooth} we see that this holds also for algebraic pseudo-convergent sequences from $V''$, which are not   fundamental sequences. It follows that the extension $V''\subset V'$ is dense using \cite[Theorem 1]{Kap}. Now, it is enough to apply Lemma \ref{k} to see that $V'$ is ind-smooth over $V''$. Indeed, let $B\subset V'$ be a finitely generated $V''$-subalgebra
and $w$ its inclusion. By separability, $w(H_{B/V''})\not = 0$ and choose an element $d\in V''$, which is standard for $B$ over $V''$. Then the composite map $B\to V'\to V'/d^3V'\cong V''/d^3V''$ factors obviously to a smooth  $ V''/d^3V''$-algebra. By Lemma \ref{k} $w$ factors through a smooth $V''$-algebra.
\hfill\ \end{proof}

\begin{Corollary} \label{c0}  For an extension $V \subset V'$ of valuation rings containing $\bf Q$ with the same value group $\Gamma\subset {\bf R}$, $V'$ is ind-smooth over $V$.
\end{Corollary}

\begin{proof}  By Proposition \ref{pass-to-hat} we may reduce to the case when $V$, $V'$ are complete and so they are Henselian since $\dim V=1$ , that is, they   contain their residue fields $k, k'$. Let $K$, $K'$ be the fraction fields of $V$, resp. $V'$. Then $V'$ is an immediate extension of  
$V''=V'\cap K(k')$ and so it is ind-smooth by the above proposition. Express $k'$ as a filtered union of some finitely generated field extensions $(k_i)$ of $k$. It is enough to see that $V_i=V'\cap K(k_i)$ is an ind-smooth extension of $V$. But $V_i$ is even essentially smooth over $V$ because $k_i$ is so over $k$.   
\hfill\ \end{proof}

\section{Extensions of valuation rings}

The following proposition is an extension of Corollary
 \ref{c0}.
\begin{Proposition} \label{p} Let $V\subset V'$ be an extension of valuation rings containing $\bf Q$. Suppose  $\dim V<\infty$ and
the value group extension of $V\subset V'$ is trivial. 
Then $V'$ is ind-smooth over $V$.
\end{Proposition}

\begin{proof} 

 For the proof we apply Lemma \ref{factor-through}.

 Let   $E $ be  a $V$-algebra of finite presentation, let us say $E \cong
V[Y]/I$, $Y =
(Y_1\ldots, Y_m)$, $I$ being  a finitely generated ideal. Let $w:E\to V'$ be a $V$-morphism. We will show that $w$ factors through a smooth $V$-algebra.
 $E$ is finitely generated and so it is $\Im w$. By Lemma \ref{fp-lemma} we see that $\Im w$ is finitely presented. So we may replace $E$ by $\Im w$, that is we may assume $w$ injective. By separability  we have  $ w( H_{E/V})\not =0$, let us assume that  $ w( H_{E/V})V' \supset zV'$ for some $z \in V $, $z\not =0$. Replacing $z$ by a power of it we may assume that $z=\sum_i^s b_ib'_i$ for some $b_i=\det(\partial f_{ij}/\partial Y_{j_i})$ for some systems of polynomials $f_i$ from $I$ and $b''_i\in V[Y]$ which kills $I/(f_i)$.
Similarly as in  \cite[Lemma 4]{ZKPP} we may
assume that we can take $s=1$, that is   for some polynomials $ f = (f_1,\ldots,f_r)$
from $I$, we have  $z\in NM E$ for some $N\in ((f):I)$ and a $r\times r$-minor $M$ of the Jacobian matrix $(\partial f_i/\partial Y_j)$ (since $V'$ is a valuation ring this reduction is much easier). Thus we may assume $z$ is standard over $V$ (see the beginning of Section 2), which is necessary later to apply Lemma \ref{k}.
 Let $q_2'\in \Spec V'$,  be the minimal prime  ideal of $zV'$ and $q_2=q_2'\cap V$. As the value group extension of $V\subset V'$ is trivial we have $q'_2=q_2V'$.

   Let $q_1\in \Spec V$, $q_1\subset q_2$ be  the greatest prime ideal of $V$  not containing $z$. Then $q_1\not = q_2$.  The extension $ V_{q_2}/q_1V_{q_2}\subset  V'_{q'_2}/q_1V_{q'_2}$  has the trivial value group extension. and so it is ind-smooth by Corollary \ref{c0}. The composite map $E\xrightarrow{w} V'\to V'_{q'_2}/q_2V'_{q'_2}$ factors by a smooth  $ V_{q_2}/q_1V_{q_2}$-algebra $G$, let us say it is the composite map $E\xrightarrow{\alpha} G\xrightarrow{\beta}  V'_{q'_2}/q_2V'_{q'_2}$. We  may assume that  $G= (V_{q_2}/q_1V_{q_2})[U]_{g'h}/(g)$, with $U=(U_1,\ldots,U_l)$, $g'=\partial g/\partial U_1$, $g, h\in V[U]$ by \cite[Theorem 2.5]{S} and let $ \beta$ be given by $U+(g)\to u+q_1 V'_{q'_2}$ for some  $u\in  (V'_{q'_2})^l$. Note that \cite[Theorem 2.5]{S} gives just that a localization of $G$ has the form a localization of $C=(V_{q_2}/q_1V_{q_2})[U]_{g'}/(g)$ and so the above composite map factors through a $C_h$ for some $h\in V[Y]$. Then $g(u)\equiv 0$ modulo $q_1 V'_{q'_2}$ and in particular   $g(u)\equiv 0$ modulo $z^3 V'_{q'_2}$.  Then $g(u)=z^3t$ for some $t\in V'_{q'_2}$. Note that the composite map $E\to V'\to V'_{q'_2}$ factors through the smooth $V_{q_2}$-algebra  $D=(V_{q_2}[U,T]/(g-z^3T))_{g'h}$ modulo $z^3$, where $T\to t$. By Lemma \ref{k} we see that $E\to  V'_{q'_2}$ factors through a smooth $D$-algebra $D'$ which is also smooth over $V_{q_2}$. Using Lemma \ref{delocalize}  we see that $w$ factors through a finitely presented $V$-algebra $E''$, let us say through a map $w'':E''\to V'$ with $w''(H_{E''/V})\not \subset q'_2$. More precisely, by Lemma \ref{delocalize} there exist a finitely presented $V$-algebra $E''$ and $c\in E''$ with $E''_c\otimes_V V_q\cong D'[X,X^{-1}]$ and a factorization $V\to E\to E''\to V'$ such that $E'\to V'_{q'}$ factors through $E'\to E'_c\otimes_V V_q\to V'_{q'}$.  Note that 
  $\dim V'=\dim V<\infty$  because $V,V'$ have the same value group.   We  arrive in  finite steps using induction on $\dim V'/zV'$  in the case when $z$ is a unit, that is we can embed $E$ in a smooth $V$-algebra. This is enough by Lemma \ref{factor-through}.
\hfill\ \end{proof}

\begin{Theorem}\label{T0} Let $V\subset V'$ be an immediate extension of valuation rings   containing $\bf Q$.
Then $V'$ is ind-smooth over $V$.
\end{Theorem}
\begin{proof} Let $K\subset K'$ be the fraction field of $V\subset V'$ and $K''\subset K'$  a pure transcendental extension  of $K$ generated by a transcendental basis of $K'/K$, that is $K'/K$ is algebraic.  Applying Lemma \ref{imm-ind-smooth} possible infinitely and even uncountably many, as in Proposition \ref{p0} we see that $V''=V'\cap K''$ is ind-smooth over $V$ and all transcendental pseudo-convergent sequences in $V''$ over $V$ having a pseudo limit in $V'$, which are not fundamental sequences,  have  pseudo-limits in $V''$. Thus we reduce to show that $V'$ is ind-smooth over $V$ when $K'/K$ is algebraic. Actually, it is enough to assume $K'/K$ finite because $V'$ is the filtered union of $V'\cap L$ for all subfields $L\subset K'$ which are finite extension over $K$.

Let $E=V[Y]/I$, $Y=(Y_1,\ldots,Y_n)$ be a finitely generated $V$-subalgebra of $V'$ (so finitely presented by Lemma \ref{fp-lemma}) and a map $w:E\to V'$. By  Lemma \ref{factor-through} it is enough to show that $w$ factors through a smooth $V$-algebra.  Consider $H_{E/V}$ and 
  a standard element $z\in V$ for $E$ over $V$, so $w(z)\in w(H_{E/V})V'$, as in the proof of Proposition \ref{p}. If  $V\subset V'$ is dense  we may apply  Proposition \ref{pass-to-hat} to see that $w$ factors through a smooth $V$-algebra (note that $w(H_{E/V})\not =0$ says that the composite map $E\to V'\to K'$ factors through a smooth $V$-algebra). In the remaining case the factorization is constructed in some steps:  for the standard element $z$ one chooses adjacent prime ideals $q_1\subset q$ of $V$ such that $w(z)\in qV'\setminus q_1V'$ and construct a factorization $E\to E'\xrightarrow{w'} V'$ such that $w'(H_{E'/V})\not \subset qV'$, where $E'$ is finitely presented over $V$. If after finite steps we get a factorization $E\to E^{(n)}\xrightarrow{w^{(n)}}V'$ such that $w^{(n)}(H_{E^{(n)}/V})V'=V'$ the goal is reached.  A hard problem is to show that we can find such $E^{(n)}$ in finite steps. For this we will consider a finite partition ${\mathcal P}_i$, $i=1,\ldots,s$ of $\Spec V$ corresponding to those $q\in\Spec V$ which have the same dimension $f_i=f_q\leq f=[K':K]$ of the fraction field extension $K_q\subset K'_q$ of $V/q\subset V'/qV'$. We will see that to each construction $q_1$ change  from one ${\mathcal P}_j$ to another one from  ${\mathcal P}_i$ with $j<i$ and $f_i<f_j$. Finally we arrive in finite steps to the case $f_{q_1}=1$,
which is done easily as in the dense case.

   Assume  $V\subset V'$ is not dense and $q,q_1, f_{q_1},({\mathcal P}_i)_{1\leq i\leq s}$  as above. More precisely,
  let $q'\in \Spec V'$ be the minimal prime ideal of $w(z)V'$ and $q=q'\cap V$. Thus  $qV'=q'$ because $V\subset V'$ is immediate.  Let $q_1'\in \Spec V'$ be the prime ideal corresponding to the maximal ideal of the fraction ring of $V'$ with respect to the multiplicative system  generated by $z$. Then $q'_1$ is the biggest prime ideal of $V'$ contained strictly in $q'$ and so height$(q'/q'_1)=1$. Set $q_1=q'_1\cap V$. We have $f_q\leq f_{q_1}$.   
   
    Let $x_q$ be a primitive element of the separable finite extension $K'_q/K_q$ and $g_q\in V/q[X]$ be a primitive polynomial multiple of Irr$(x_q,K_q)$ by a nonzero constant of $K$. Note that if $q,\qq\in {\mathcal P}_i$, $q\subset \qq$ then $f_q=f_{\qq}=f_i$ and $g_q$ remains irreducible over $V/\qq$.
    Clearly, $f_s=1$ because $f_{\mm}=1$ for the maximal ideal $\mm$ of $V$, the extension $V\subset V'$ being immediate.
A set  ${\mathcal P}_i$ has a maximum element for inclusion namely $\pp_i=\cup_{ \qq \in {\mathcal P}_i}\qq$. Indeed, $\pp_i$  is clearly a prime ideal and if  $f_{\pp_i}<f_i$ then $f_{\qq}<f_i$ for some $\qq\in  {\mathcal P}_i$, which is false.

    Assume $q_1\in {\mathcal P}_j$. If $q_1\not =\pp_j$ then $(V/q_1)_{\pp_j}\subset (V'/q_1V')_{\pp_jV'}$ is in fact a localization of $(V/q_1)[X]/(g_{q_1})$ because $g_{q_1}'=\partial g_{q_1}/\partial X$ corresponds to a unit in $(V'/q_1V')_{\pp_j}$ and so the composite map $E\to V'\to (V'/q_1V')_{\pp_jV'}$ factors through an etale $V/q_1$-algebra of the form $((V/q_1)[X]/(g_{q_1})_{g_{q_1}'h}$ for some $h\in V[X]$. In particular $E\to V'\to (V'/(z^3))_{\pp_jV'}$ factors  through an etale $V/(z^3)$-algebra and by Lemma \ref{k} the map  $E\to V'\to V'_{\pp_jV'}$ factors through a smooth $V$-algebra. Using Lemma \ref{delocalize}  we see that $w$ factors through a finitely presented $V$-algebra $E'$, let us say through a map $w':E'\to V'$ with $w'(H_{E'/V})\not \subset \pp_jV'$. Changing $E$ by $E'$ we see that the new $q$ belongs to ${\mathcal P}_i$ for some $i>j$. Moreover, the new $q_1$ belongs also to ${\mathcal P}_i$ for some $i>j$, because otherwise we get $q_1=\pp_j$. 
    
    If $q_1=\pp_j$ then $q\in {\mathcal P}_{j+1}$ and we apply  Corollary \ref{c0}.  Then   we see that   $(V/q_1)_q\subset (V'/q_1')_{q'}$ is ind-smooth and as above we see that the composite map $E\to V'\to (V'/q_1')_{q'}$ factors through a
 smooth $V/q_1$-algebra and finally by Lemmas \ref{k} and \ref{delocalize} we get that   $w$ factors through a finitely presented $V$-algebra $E'$, let us say through a map $w':E'\to V'$ with $w'(H_{E'/V})\not \subset qV'$. Now the new $q_1$,
 that is the old $q$, belongs to ${\mathcal P}_{j+1}$.  In some steps (at most $s$) we arrive to the case when $f_{q_1}=1$  
    
   If $f_{q_1}=1$  then we get   $f_{q''_1}=1$  for all  $q''_1\in \Spec V$ containing  $q_1$. Actually, we get $V/q= V'/q'$ and so in particular    $V/q\subset V'/q'$ is ind-smooth.   Using Lemma \ref{k} we see that $w$ factors through a smooth (in fact etale) $V$-algebra.
  Applying  Lemma \ref{factor-through} we are done.   
\hfill\ \end{proof}

\begin{Proposition} \label{p1}   Let $V\subset V'$ be an extension of  valuation rings. Suppose that

\begin{enumerate}

\item $V$ is a discrete valuation ring extending ${\bf Z}_{(p)}$ with $\pi$ its local parameter, and $p$ a prime number.

\item $\pi V'$ is the maximal ideal of $V'$,

\item the residue field extension of $V\subset V'$ is separable.
\end{enumerate}

Then $V\to V'$ is ind-smooth.
\end{Proposition}

\begin{proof}  Let $E$, $w$, $H_{E/V}$, $z$  be as in Proposition \ref{p} and we may assume $K'$ is the fraction field of $\Im w$. choose $q'_2\in \Spec V'$ a minimal prime ideal of $w(H_{E/V})V'$. If $q'_2\not =\pi V'$ then using Zariski's Uniformization Theorem  we may change $E,w$ with some $E',w'$ such that $w'(H_{E'/V})V'\not \subset q'_2$. Step by step we arrive to the case when either $w(H_{E/V})V'=V'$, or $w(H_{E/V})V'$ is a $\pi V'$-primary ideal. In the first case, $w$ factors through a localization of $E$ which is smooth.  In the second case, $q'_1=\cap_{i\in {\bf N}} \pi^i V'$ is a prime ideal and the  composite map $V\to V'\to V'/q'_1$ is a regular map of discrete valuation rings and so an ind-smooth map by the classical N\'eron Desingularization. The proof ends by using  Lemma \ref{factor-through}.
 \hfill\ \end{proof}
 
\begin{Corollary} \label{C1} Let $V$ be a discrete valuation ring  extending ${\bf Z}_{(p)}$ with  $p$ a prime number and $V'$ an ultrapower of $V$ with respect to a nonprincipal ultrafilter on $\bf N$. Then $V\subset V'$ is ind-smooth.
\end{Corollary}
For the proof note that the maximal ideal of $V$ generates the maximal ideal of $V'$ and apply the above proposition.
 
 \begin{Proposition} \label{p2} Let $V$ be a discrete valuation ring  extending ${\bf Z}_{(p)}$ with  $p$ a prime number and $V\subset V'$ an extension of valuation rings such that 
 \begin{enumerate}
 \item $p$ is a local parameter of $V$,
 \item   $p V'$ is a $\mm'$-primary ideal of $V'$, where $\mm'$ is the maximal ideal of $V'$,

\item the residue field extension of $V\subset V'$ is separable.
\end{enumerate}
Then $V'$ is a filtered direct limit of regular local rings essentially of finite type over $V$. 
 \end{Proposition}
 
 \begin{proof} As in Proposition \ref{p1} we may consider $w:E\to V'$ and we may reduce to the case when  
  $w'(H_{E'/V})V'$ is $\mm'$-primary ideal. We may assume that $p^s$ is a standard for $E$ over $V$ for some $s\in {\bf N}$ and as in the proof of \cite[Theorem 3.6]{P2'} there exists a local essentially smooth $V$-algebra $G$ and $b\in G$ such that the map $E/p^{3s}E\to V'/p^{3s}V'$ factors through $G/(p^{3s},p-b)$. Then a variant of Lemma \ref{k} in the idea of \cite[Proposition 3.4]{P2'} shows that $w$ factors through a local essentially smooth $D=G/(p-b)$-algebra $D'$. This $D'$ is  regular local since $D$ is so. Now apply Lemma \ref{factor-through}.  
  \hfill\ \end{proof}
 
\vskip 0.5 cm

\section{Structure of equicharacteristic valuation rings possessing a cross-section}

Modulo all the reductions and simplifications that go into the overall proof of Theorem \ref{m}, our ultimate source of expressions of valuation rings as filtered direct limits of smooth rings is Lemma \ref{complete-lem} below. This lemma describes some valuations on an affine space for which local uniformizations can be constructed by successively blowing up regular centers as in \cite[4.5, 4.19]{Sha} following Perron's algorithm (whose relevance to the resolution of singularities was explained already in \cite{Zar}). We present a more direct argument for this uniformization that is close to \cite[Lemma 4.6]{Po1} and rests on the following lemma that captures the ``combinatorial'' part of local uniformization.

We will need the following lemma (see \cite[2.2]{Ell}, or \cite[4.6.1]{Po1}, or \cite[6.1.30]{GR}).

\begin{Lemma} \label{EPGR-lem}
For a totally ordered abelian group $\Gamma$, the submonoid $\Gamma_{\geq 0} \subset \Gamma$ of nonnegative elements is a filtered  union of its finite free submonoids isomorphic to ${\bf Z}_{\geq 0}^r$,  where $r \in { \bf Z}_{\geq 0}$ need not be constant. 
\end{Lemma}

We include a mixed characteristic version of  the following lemma because it requires virtually no additional effort in comparison to the equicharacteristic case  that we will use below.

\begin{Lemma} \label{complete-lem} 
\begin{enumerate}
\item 
For a field $\bf F$, a valuation ring ${\bf F} \subset V$ with fraction field ${\bf F}(x_1, \dotsc, x_n)$ such that $\val(x_1)$, $\dots$, $\val(x_n)$ are $\bf Z$-linearly independent is a countable direct union of essentially smooth~$\bf F$-algebras.

\item 
For a discrete valuation ring $\Lambda$ with uniformizer $\pi$ and fraction field $\bf F$, a valuation ring $\Lambda \subset V$ that dominates $\Lambda$ and has fraction field ${\bf F}(x_1, \dotsc, x_n)$ such that $\val(\pi), \val(x_1), \dots, \val(x_n)$ are $\bf Z$-linearly independent is a countable direct union of regular local $\Lambda$-algebras of the form
\[
 (\Lambda[Y_1, \dots, Y_{n + 1}]/(\pi - Y_1^{b_1}\cdots Y_{n + 1}^{b_{n + 1}}))_{(Y_1,\, \dotsc,\, Y_{n + 1})} \ \ \mbox{with} \ \ \gcd(b_1, \dotsc, b_{n + 1}) = 1.
\]
\end{enumerate}
\end{Lemma}

\begin{proof}
To avoid repeating the argument, we will prove both claims simultaneously, so in (1) we set $\Lambda = {\bf F}$ and $\pi = 0$  and in both parts we set $p = $Char$(\Lambda/(\pi))$. By \cite[Theorem 1 in VI (10.3)]{Bou}, 
\[
\gamma_1 = \val(x_1), \ \  \dots, \ \ \gamma_n = \val(x_n), \ \ \gamma_{n + 1} = \val(\pi) \ \ (\mbox{resp. ~$\gamma_{n + 1}=  0$ if $\pi = 0$})
\]
satisfy $\Gamma \cong {\bf Z} \gamma_1 \oplus \cdots \oplus {\bf Z}\gamma_{n + 1}$, where $\Gamma$ is the value group of $V$.
We set $N = n + 1$ (resp.,~$N = n$ if $\pi = 0$) and use Lemma \ref{EPGR-lem} to find a countable sequence $\Gamma_0 \subset \Gamma_1 \subset \dots$ of submonoids of $\Gamma_{\geq 0}$ with $\Gamma_i \simeq {\bf Z}^N_{\geq 0}$ for each $i$ and $\Gamma_{\geq 0} = \bigcup_{i \geq 0} \Gamma_i$. We fix a ${\bf Z}_{\geq 0}$-basis $\nu_{i1}, \dots, \nu_{i N}$ of $\Gamma_i$ with $(\nu_{01}, \dots, \nu_{0N}) = (\gamma_1, \dots, \gamma_N)$, so that the elements $\nu_{i1}, \dots, \nu_{i N}$ are $\bf Z$-linearly independent in $\Gamma$, and we express them in terms of the fixed $\bf Z$-basis:
\[
\nu_{ij} =  d_{ij 1 } \gamma_1 + \dots + d_{ij N } \gamma_N \ \ \mbox{for unique} \ \ d_{ij 1 }, \dots, d_{ij N } \in {\bf Z} \ \ \mbox{and every} \ \ j = 1, \dots, N.
\]
We set $x_{n + 1} = \pi$ and note that, by  construction, for each $i \geq 0$ and $1 \leq j \leq N$, the element

$y_{ij} = x_1^{d_{ij1}} \cdots x_N^{d_{ijN}} \in {\bf F}(x_1, \dots, x_n) \ \ \mbox{has valuation} \ \ \nu_{ij}.$

Since $\Gamma_{i'} \subset \Gamma_{i}$ for $i' < i$, each $y_{i'j}$ is in a unique way a monomial in the elements $y_{i1}, \dotsc, y_{iN}$: 
\[
\mbox{if we express} \ \  \nu_{i'j} = b_{i'i1} \nu_{i1} + \dots + b_{i'iN} \nu_{iN} \ \ \mbox{with}\ \
 b_{i'ij} \in {\bf Z}_{\geq 0}, \ \ \mbox{then} \ \ y_{i'j} = y_{i1}^{b_{i'i1}} \cdots y_{iN}^{b_{i'iN}}.
\]
Since the valuations of $y_{i1}, \dots, y_{iN}$ are $\bf Z$-linearly independent, the $\Lambda$-subalgebra $\Lambda[y_{i1}, \dots, y_{iN}]$ of ${\bf F}(x_1, \dots, x_n)$ is the regular ring 
\[
\Lambda[Y_{i1}, \dots, Y_{iN}]/(\pi - Y_{i1}^{b_{1}} \cdots Y_{iN}^{b_{N}}) \ \ \mbox{with} \ \ b_i = b_{0Ni} \ \  (\mbox{resp.}\ \ \Lambda[Y_{i1}, \dots, Y_{iN}] \ \ \mbox{if} \ \ \pi = 0),
\]
where $\gcd(b_{1}, \dots, b_{N}) = 1$ because $\val(\pi)$ is assumed to be a primitive element of $\Gamma$ when $\pi \neq 0$. In particular, we obtain a nested sequence of $\Lambda$-subalgebras 
\[
R_i = \Lambda[y_{i1}, \dots, y_{iN}]_{(y_{i1},\, \dots,\, y_{iN})} \subset V
\]
that are regular (resp.,~essentially smooth if $\pi = 0$) and it remains to argue that every $f \in V$ belongs to some $R_i$. For this, we first express $f$ as a rational function as follows:
\[
 f = (\sum \lambda_{s_1,\, \dots,\, s_N} x_1^{s_1}\cdots x_N^{s_N})/( \sum \lambda'_{r_1,\, \dots,\, r_N} x_1^{r_1} \cdots x_N^{r_N}) \ \ \mbox{with} \ \ \lambda_{s_1,\, \dots,\, s_N}, \lambda'_{r_1,\, \dots,\, r_N} \in \Lambda^\times \cup \{ 0 \}.
\]
The linear independence of the $\gamma_i$ ensures that the valuations of the monomials that appear in the numerator (resp.,~denominator) are all distinct. Thus, by taking out the monomials with minimal valuations, we reduce to showing that every $x_1^{\alpha_1} \cdots x_N^{\alpha_N}$ with $\alpha_j \in {\bf Z}$ and $\alpha_1\gamma_1 + \dots + \alpha_N\gamma_N > 0$ is a product of nonnegative powers of the elements $y_{i1}, \dots, y_{iN}$ for some $i \geq 0$. For this, it suffices to note that $\alpha_1\gamma_1 + \dots + \alpha_n\gamma_n$ lies in some $\Gamma_i$, and then to express it as a ${\bf Z}_{\geq 0}$-linear combination of the $\nu_{i1}, \dots, \nu_{i N}$: more precisely, if $\alpha_1\gamma_1 + \dots + \alpha_N\gamma_N = c_1 \nu_{i1} + \dots + c_N \nu_{iN}$ with $c_j \in {\bf Z}_{\geq 0}$, then 
\[
x_1^{\alpha_1} \cdots x_N^{\alpha_N} = y_{i1}^{c_1}\cdots y_{iN}^{c_N}. 
\]
\hfill\ \end{proof}

For a valuation ring $V$ with the value group $\Gamma$ and the fraction field $K$, a \emph{cross-section} of $V$ is a section 
\[
s \colon \Gamma \to K^* \ \ \mbox{in the category of abelian groups of the valuation map} \ \ \val \colon K^* \to \Gamma.
\]
(see for details in the Appendix).

\begin{Proposition} \label{CS-benefits}
An equicharacteristic valuation ring $V$ that has a cross-section $s \colon \Gamma \to K^*$ and a subfield $k \subset V$ lifting the residue field is an immediate extension 
\[
 V_0 = \bigcup_i V_i \subset V
\]
of a filtered  union of valuation subrings $V_i \subset V$ dominated by $V$ such that each $V_i$ has a finitely generated value group, is a countable increasing  union of localizations of smooth $k$-subalgebras of $V$ {so $V_0$ is ind-smooth over $k$}, and has the restriction of $s$ as a cross-section.
\end{Proposition}

\begin{proof}
By Lemma \ref{EPGR-lem}, the submonoid $\Gamma_{\ge 0} \subset \Gamma$ of positive elements is a filtered  union $\Gamma_{\geq 0} = \bigcup_{i} \Gamma_i$ of submonoids $\Gamma_i \simeq {\bf Z}_{\geq 0}^{d_i}$ with $d_i \geq 0$. Thus, the cross-section $s$ gives rise to the filtered  system of subfields $k_i = k(s(\gamma)\, \vert\, \gamma \in \Gamma_i)$ of the field of fractions $K$ of $V$. By choosing a ${\bf Z}_{\geq 0}$-basis for $\Gamma_i$ and applying \cite[Theorem 1, in VI section 10]{Bou} we see that each $k_i$ is a purely transcendental extension of $k$ and that the value group of the valuation subring $V_i = V \cap k_i$ of $V$ is ${\bf Z}^{d_i} \simeq {\bf Z}\Gamma_i \subset \Gamma$. By construction, $s$ restricts to a cross-section of $V_i$ and, by Lemma \ref{complete-lem}, each $V_i$ is a filtered  union of localizations of $k$-subalgebras. The construction ensures that $V$ is an immediate extension of the resulting $V_0$.
\end{proof}

 \section{ Counterexamples when the value groups are finitely generated}

\begin{Lemma}\label{Lemma 5.1} Let $V \subset V'$ be an extension of valuation rings which is ind-smooth. Then $\Omega_{V'/V}$, that
is $ H_0(V, V',V')$ in terms of Andre-Quillen homology, is a flat $V'$-module and $H_1(V, V',V')
= 0$ (the
last homology is usually  denoted by $\Gamma_{V'/V}$ ).
\end{Lemma}
\begin{proof} Assume that $V'$
is the filtered direct limit of some smooth $V$ -algebras $B_i$, $i \in I$. Then
$\Omega_{B_i/V}$ is projective over $B_i$ and $H_1(V, B_i,B_i)= 0$ by e.g \cite[Theorem 3.4]{S}. But $\Omega_{V'/ V}$ , and $H_1(V, V',V')$
 are filtered direct limits of $V'\otimes_{B_i}\Omega_{B_i/V}$
resp. $V\otimes_{B_i}
H_1(V, B_i,B_i)$
 by \cite[Lemma
3.2]{S}, which is enough. 
\hfill\ \end{proof}

\begin{Lemma}\label{l 5.1}

Let $V\subset V'$ be an extension of valuation rings of dimension one with the same residue field and let  $\Gamma\subsetneq \Gamma'$ be their value group extension. Assume that  $\Gamma'/\Gamma$  has  torsion.
Then the extension $V\subset V'$ is not ind-smooth.
\end{Lemma}

\begin{proof} Let $\gamma\in \Gamma'\setminus \Gamma$ be such that $n\gamma\in \Gamma$ for some positive integer $n$. Choose an element $x\in V'$ such that $\val(x)=\gamma$. Then $x^n=zt$ for some $z\in V$ and an unit $t\in V'$. Thus the system $S$ of polynomials $X^n=zT$, $TT'=1$ over $V$ has a solution in $V'$. If $V'$ is ind-smooth over $V$ then $S$ has a solution in a smooth $V$-algebra and so one $({\tilde x},{\tilde t},{\tilde t}')$ in the completion of $V$. But then $\gamma =\val(z)/n=\val({\tilde x})$ must be in $\Gamma$ which is false.
\hfill\ \end{proof}

\begin{Lemma}\label{l 5.2}

Let $V\subset V'$ be an extension of valuation rings of dimension one containing $\bf Q$ having the same residue field $k$. Assume that $V$ contains $k$ and  its value group  $\Gamma\subset {\bf R}$ is   dense in $\bf R$. Also assume that the value group $\Gamma'\subset {\bf R}$ of $V'$ is finitely generated (that is finitely generated over $0$), $\Gamma\not =\Gamma'$ and $\Gamma'/\Gamma$  has no torsion.
Then the extension $V\subset V'$ is not ind-smooth.
\end{Lemma}

\begin{proof}
 Since $\Gamma$ is free over $\bf Z$ we may take  a basis of positive elements  $\gamma_1,\ldots,\gamma_m$  of $\Gamma$ which may be completed with some positive elements $\gamma_{m+1},\ldots,\gamma_n\in \Gamma'$  to a basis of $\Gamma'$. 
Choose $x_1,\ldots,x_m\in V$ and $x_{m+1},\ldots,x_n\in V'$ such that $\val(x_i)=\gamma_i $.  Let $V_0=V\cap k(x_1,\ldots,x_m)$ and $V'_0=V'\cap k(x_1,\ldots,x_n)$. 

We will show that $\Omega_{V'_0/V_0}$ has torsion. {\bf First assume} that $n=m+1$. We will use the  proof of \cite[Lemma 7.2]{Po1}. By Lemma \ref{EPGR-lem} $\Gamma_+=\cup_{j\in {\bf N}} \Gamma_j$ for some monoids $\Gamma_j\subset \Gamma_+$ generated by bases of $\Gamma$, the union being filtered. We consider as in the quoted lemma two real sequences $(u_i)$, $(v_i)$  which converge in $\bf R$ to $\gamma_n$ and such that

1) $u_j,v_j\in \Gamma_j$ and $u_{j+1}-u_j, v_j-v_{j+1}\in \Gamma_{j+1}$,

2) $v_j-u_j$ is an element of the  basis $\nu_j$ of $\Gamma$ generating $\Gamma_j$, we may assume 
$v_j-u_j=\nu_{j1}.$

3) $u_j<\gamma_n<v_j$ for all $j$.

 We may also suppose that $u_{j+1}-u_j=v_j-v_{j+1}$ if necessary restricting to a subsequence of $(\Gamma_j)$. Let   $a_j$, $b_j$ be in $V$ with values $u_j$, resp. $v_j$,  and take $y_{jn}=x_n/a_j$ and $z_{jn}=b_j/x_n$ in $V'$. As in the proof of Lemma 4.2 a),
we have $\nu_{ji}=d_{ji1}\gamma_1+\ldots +d_{jim}\gamma_m$ and set $y_{ji}=x_1^{d_{ji1}}\cdots x_m^{d_{jim}}\in V$ which has valuation $\nu_{ji}$, $i\in [m]$. Then $V_0$ is a filtered  union of localizations $B_j$ of $k[y_{j1},\ldots,y_{jm}]$
 and $V'_0$ is a filtered  union  of localizations $C_j$ of $B_j[Z_j,Z'_j]/(Z_jZ'_j-y_{j1})\cong B_j[z_j,z'_j]$, where   $z_j=x_n/a_j$ and $z'_j=b_j/x_n$ in $V'$. Note that the map $C_j\to C_{j+1}$ is given by $Z_j\to (a_{j+1}/a_j)Z_{j+1}$, $Z'_j\to (b_j/b_{j+1})Z'_{j+1}$.

We claim that the map $f_j:C_{j+1}\otimes_{C_j} \Omega_{C_j/B_j}\to  \Omega_{C_{j+1}/B_{j+1}}$  given by $d z_j\to (a_{j+1}/a_j)d z_{j+1}$, $d z'_j\to (b_j/b_{j+1})dz'_{j+1}$ is injective. Indeed, an element from $\Ker f_j$ induced by $w=\alpha\otimes dz_j+\beta \otimes dz'_j$, $\alpha, \beta\in C_{j+1}$ must go by $f_j$ in
 $$\alpha (a_{j+1}/a_j)dz_{j+1}+\beta (b_j/b_{j+1})dz'_{j+1}\in <z'_{j+1}dz_{j+1}+z_{j+1}dz'_{j+1}>$$
 
\noindent in $C_{j+1}dz_{j+1}\oplus C_{j+1}dz'_{j+1}$. So $\alpha (a_{j+1}/a_j)=\mu z'_{j+1}$ and $\beta (b_{j}/b_{j+1})=\mu z_{j+1}$ for some $\mu \in C_{j+1}$. It follows that 

$$w=(\mu (b_{j+1}/b_j)(a_{j+1}/a_j))\otimes z'_jdz_j+(\mu (b_{j+1}/b_j)(a_{j+1}/a_j))\otimes z_jdz'_j$$

\noindent belongs to $<z'_jdz_j+z_jdz'_j>$, which shows our claim.

We may assume that $\val( z_j) \leq \val( z'_j)$. For some $j' \geq j$ we have $z'_j=t_jz_j$ in $C_{j'}$ for some $ t_j$ in $C_{j'}$. We  have $dz_j, dz'_j$ in $\Omega_{C_{j'}/B_{j'}}$ and
$z_jw'_j=0$, for $w'_j=dz'_j+t_jdz_j$. So $C_{j'}\otimes_{C_j} \Omega_{C_j/B_j}$ is not torsion free. Here we should point out that the localizations are given by elements  from $\rho+ ((y_{j'i})_i, z_{j'},z'_{j'})$, $\rho\in k$, $\rho\not =0$,
 which cannot kill $w'_j$. Since $f_j$ are injective we see that $\Omega_{V'_0/V_0}$ has torsion.

Now {\bf  assume} that $n>m+1$ and consider $V''_0=V'\cap k(x_1,\ldots,x_{n-1})$. Apply induction on $n-m$, the case $n-m=1$ having been considered above. By induction hypothesis, we assume that $\Omega_{V''_0/V_0}$  has torsion. As above $V''_0$ is a filtered direct limit of some localizations ${\tilde B}_j$ of $k[{\tilde y}_{j1}, \ldots, {\tilde y}_{j,n-1}]$ and $V'_0$ is the filtered direct limit of some localizations
${\tilde C}_j$ of ${\tilde B}_j[{\tilde z}_j,{\tilde z}'_j]$. Set $I_j=(Z_jZ'_j-{\tilde y}_{j1})\subset {\tilde B}_j[Z_j,Z'_j]$. By definition we have the following exact sequence

$$0\to H_1({\tilde B}_j,{\tilde C}_j,{\tilde C}_j)\to I_j/I_j^2\xrightarrow{d} {\tilde C}_jd{\tilde z}_j\oplus {\tilde C}_jd{\tilde z}'_j\to \Omega_{{\tilde C}_j/{\tilde B}_j}\to 0.$$
If $h\in I_j$ induces an element in  $\Ker d$ then we get
 $$(\partial h/\partial Z_j)({\tilde z}_j,{\tilde z}'_j)d{\tilde z}_j\oplus (\partial h/\partial Z'_j)({\tilde z}_j,{\tilde z}'_j)d{\tilde z}'_j=0.$$
But $h={\tilde h}(Z_jZ'_j-{\tilde y}_{j1})$ for some ${\tilde h}\in {\tilde B}_j[Z_j,Z'j]$ and it follows $ {\tilde h}({\tilde z}_j,{\tilde z}'_j){\tilde z}'_j=0$, that is $ {\tilde h}({\tilde z}_j,{\tilde z}'_j)=0$ and so ${\tilde h}\in I_j$. Hence $d$ is injective and  $ H_1({\tilde B}_j,{\tilde C}_j,{\tilde C}_j)=0$.

In the Jacobi-Zariski sequence (\cite[Theorem 3.3]{S} applied to $V_0\to V''_0\to V'_0$ 
 $$0=H_1(V''_0,V'_0,V'_0)\to V'_0\otimes_{V''_0} \Omega_{V''_0/V_0}\xrightarrow{\lambda} \Omega_{V'_0/V_0}\to  \Omega_{V'_0/V''_0}\to 0$$
we see that the map $\lambda$ is injective. It follows that $\Omega_{V'_0/V_0}$ has torsion which proves our claim.

 By Proposition \ref{p0} the immediate extension $V_0\subset V$ is ind-smooth. Assume, aiming for contradiction, that $V'$ is ind-smooth over $V$. Then $V'$ is ind-smooth over $V_0$  and by the above lemma we get $\Omega_{V'/V_0}$ flat over $V'$. Again by Proposition \ref{p0} we have $V'$ ind-smooth over $V'_0$. As in the above lemma
 we obtain that  $\Omega_{V'/V'_0}$ is a flat module over $V'$. In the Jacobi-Zariski sequence  applied to $V_0\to V'_0\to V'$
  $$H_1(V'_0,V',V')\to V'\otimes_{V'_0} \Omega_{V'_0/V_0}\to \Omega_{V'/V_0}\to  \Omega_{V'/V'_0}\to 0$$

\noindent we have $H_1(V'_0,V',V')=0$ 
and the last two modules are flat by the above lemma. We obtain that  $V'\otimes_{V'_0} \Omega_{V'_0/V_0}$ is flat also over $V'$ and so $ \Omega_{V'_0/V_0}$ is also flat, which is not possible because it has torsion.  Thus $V'$ is not ind-smooth over $V$.
\hfill\ \end{proof}

\begin{Remark} \label{rem} {\em In the above proof the main point was to show that when $\Gamma'/\Gamma \not = 0$ has no torsion then $\Omega_{V'_0/V_0}$ has torsion.} 
\end{Remark}
\begin{Lemma}\label{different residue}

Let $V\subset V'$ be an extension of valuation rings of dimension one containing $\bf Q$. Assume that $V$ contains its residue field and  its value group  $\Gamma\subset {\bf R}$ is   dense in $\bf R$. Also assume that the value group $\Gamma'\subset {\bf R}$ of $V'$ is finitely generated  and $\Gamma'/\Gamma$, $\Gamma\not =\Gamma'$  has no torsion.
Then the extension $V\subset V'$ is not ind-smooth.
\end{Lemma}

\begin{proof}
In the proof of Lemma \ref{l 5.2} take $W=V'\cap \Frac(V)(x_{m+1},\ldots,x_n)$. Then the extension $V\subset W$ has the same residue field but the value group extension is $\Gamma\subset \Gamma'$. Then $\Omega_{W/V}$ has torsion as in the proof of the quoted lemma. In the Jacobi-Zariski sequence applied to $V,W,V'$
$$H_1(W,V',V')\to V'\otimes_W\Omega_{W/V}\to \Omega_{V'/V} $$
we see that the left module is zero  because the valuation extension $W\subset V'$ is ind-smooth (see Lemma \ref{Lemma 5.1}), having the same value group (see Corollary \ref{c0}). It follows that $\Omega_{V'/V} $ has torsion. But if $V\subset V'$ is ind-smooth then $\Omega_{V'/V}$ is torsion free. So  $V\subset V'$ is not ind-smooth.
\hfill\ \end{proof}

\begin{Lemma} \label{l 5.3}

Let $V\subset V'$ be an extension of valuation rings of dimension one containing $\bf Q$ with value groups $\Gamma\subset \Gamma'$ and having the same residue field $k$. Assume that $V$ contains $k$ and  its value group  $\Gamma\subset {\bf R}$ is   dense in $\bf R$. Also assume that the value group $\Gamma'/\Gamma\not =0$ has no torsion and $V'$ has a cross-section $s:\Gamma'\to K'^*$ such that $s(\Gamma)\subset K^*$, $K,K'$ being the fraction fields of $V$, resp. $V'$.
Then the extension $V\subset V'$ is not ind-smooth.
\end{Lemma}

\begin{proof} We follow the proof of the above lemma. Take $V_0=V\cap k(s(\Gamma))$ and $V'_0=V'\cap k(s(\Gamma'))$. For every finitely generated $\Gamma_1\subset \Gamma$ and $\Gamma_1'\subset \Gamma'$ such that $\Gamma_1\subset \Gamma_1'\not\subset  \Gamma$  we see as in the above lemma that for $V_{10}=V\cap k(s(\Gamma_1))$, $V'_{10}=V'\cap k(s(\Gamma'_1))$ we have a torsion in $\Omega_{V'_{10}/V_{10}}$. Then as in the above lemma we get a torsion in $\Omega_{V'_0/V_0}$ by \cite[Lemma 3.2]{S} and so in $\Omega_{V'/V}$ which implies that $V'$ is not ind-smooth over $V$ by Lemma \ref{l 5.1}.
\hfill\ \end{proof}
 
\begin{Lemma}\label{l 5.4}

Let $V\subset V'$ be an extension of valuation rings of dimension one containing $\bf Q$ with value groups $\Gamma\subset \Gamma'$. Assume that $V$ contains its residue field and  its value group  $\Gamma\subset {\bf R}$ is   dense in $\bf R$. Also assume that  the  group $\Gamma'/\Gamma\not =0$ has no torsion and $V'$ has a cross-section $s:\Gamma'\to K'^*$ such that $s(\Gamma)\subset K^*$, $K,K'$ being the fraction fields of $V$, resp. $V'$. 
Then the extension $V\subset V'$ is not ind-smooth.
\end{Lemma}
The proof follows as in Lemma \ref{l 5.3}  using now Lemma \ref{different residue}.

\vskip 0.5 cm
\section{The case when the value group is not finitely generated}

A weaker form of Theorem \ref{z} is given below with an independent proof (note that in the proof of Theorem \ref{T0} we do not use the Zariski's Uniformization Theorem.

 \begin{Theorem} \label{z1} Every valuation ring $V$ containing its residue field $k$ of characteristic zero is a filtered direct limit of smooth $k$-algebras $(R_i)_i$, that is the inclusion $k\subset V$ is ind-smooth  (in particular, all the $R_i$ are regular rings).
\end{Theorem}  
\begin{proof}  Let $\Gamma$ be the value group of $V$,  $K$ the fraction field of $V$ and $\tilde k$, $\tilde \Gamma$, $\tilde V$, ${\tilde K}$, ${\tilde s}:{\tilde \Gamma}\to {\tilde K}^*$ be given as in   Theorem \ref{big-tower}. Note that in the proof of Theorem \ref{big-tower} the ultrapoduct $k_1\subset V_1$ of $k\subset V$ gives an inclusion in $V_1$ of its residue field.  More precisely, given a map $f:A\to B$ the ultrapower of $f$ is the map between the ultrapowers of $A$ and $B$ given by $[a_i]\to [f(a_i)]$. By induction we see that the ultraproduct $k_n\subset V_n$ of $k_{n-1}\subset V_{n-1}$ gives an inclusion in $V_n$ of its residue field.  Thus ${\tilde V}=\cup_{n\in {\bf N}}  V_n$ contains its residue field ${\tilde k}=\cup_{n\in {\bf N}} k_n$.

 By Proposition \ref{CS-benefits} we see that $\tilde V$ is an immediate extension of a valuation ring ${\tilde V}_0$ which is a filtered  union of localizations of smooth ${\tilde k}$-algebras. By Theorem \ref{T0}  we get 
  ${\tilde V}_0\subset {\tilde V}$    ind-smooth. Hence  $ k\subset {\tilde V}$ is ind-smooth because $k\subset {\tilde k}$ is separable and so ind-smooth (see Lemma \ref{ind-sm-comp}). It follows that $k\subset V$ is ind-smooth  too. Indeed, let $E=k[Y]/I$, $Y=(Y_1,\ldots,Y_n)$ be a finitely generated $k$-algebra and $w:E\to V$ be a morphism of $k$-algebras. For the result we will apply Lemma \ref{factor-through} if we show that $w$ factors through a smooth $k$-algebra. But the composite map $E\to V\to {\tilde V}$ factors through a smooth $k$-algebra
 $S=k[Z]/(h)$, $Z=(Z_1,\ldots,Z_s)$ for a system of polynomials $h$ from $ k[Z]$, thus it is the composite map $E\xrightarrow{t} S \xrightarrow{\tilde w} {\tilde V}$,
 where $\tilde w$ is given by $Z\to {\tilde z}\in {\tilde V}^s$ and $t$ is induced by $Y\to g\in k[Z]^n$. Then $\tilde z$ is a solution of $h$ and of the system $P$ of polynomial equations $g(Z)=w(Y)$ over $V$. Actually, $\tilde z$ belongs to some $V_n$ and so $h, P$ has also a solution $z_{n-1}$ in $V_{n-1}$ because $V_n$ is an ultrapower of $V_{n-1}$. By induction we get a solution $z\in V$ of $h,P$. Therefore, $w$ factors through the map $S\to V$, $Z\to z$. 
\hfill\ \end{proof}

\begin{Lemma}\label{l 5.5}
Let $B$ be an $A$-algebra and $A^n$, $B^n$ be the product of $n$-copies of $A$, resp. $B$. Then $\Omega_{B^n/A^n}\cong \Omega_{B/A}^n$ and $H_1(A^n,B^n,B^n)\cong H_1(A,B,B)^n$ .
\end{Lemma}

\begin{proof} 
We treat only the case $n=2$. If $B=A[X]/I$, $X=(X_i)_i$ then $B^2=A^2[X]/J$, where $J$ is given by polynomials of the form $h_{f,g}=\sum_{j=(j_1,\ldots,j_s)} (a_j,b_j)X^j$ for some polynomials $f=\sum a_jX_j$, $g=\sum b_jX^j$ from $I$. Then $\Omega_{B^2/A^2}$  is the cokernel of the map $d:J/J^2\to \oplus_i B^2 dX_i$ given by  $h_{f,g}\to \sum_i(\partial f/\partial X_i,\partial g/\partial X_i ) dX_i=\sum_i(\partial f/\partial X_i dX_i,\partial g/\partial X_i  dX_i)$. Also $H_1(A^2,B^2,B^2)$ is the kernel of $d$ and note that $d(h_{(f,g)})=0$ if and only if $d_0(f)=0$ and $d_0(g)=0$, $d_0$ being the  map $I/I^2\to \oplus_i B dX_i$ .  Hence  $\Omega_{B^2/A^2}\cong  \Omega_{B/A}^2$ and 
 $H_1(A^2,B^2,B^2)\cong H_1(A,B,B)^2 $. 
\hfill\ \end{proof}

\begin{Lemma}\label{l 5.6} Let $B$ be an $A$-algebra, $\mathcal{U}$ and ultrafilter on a set $U$ and $\tilde B$, resp. $\tilde A$ the ultrapowers of $B$ (see the Appendix for the details), resp. $A$ with respect to $\mathcal{U}$. Then $\Omega_{{\tilde B}/{\tilde A}} $ (resp. $H_1({\tilde A},{\tilde B}, {\tilde B})$) is the corresponding ultrapower of $\Omega_{B/A}$  (resp. $H_1(A,B,B)$) with respect to  $\mathcal{U}$. In particular,
 $\Omega_{B/A}$ has torsion if and only if $\Omega_{{\tilde B}/{\tilde A}} $ has torsion and $H_1({\tilde A},{\tilde B},{\tilde B})=0$ if and only if $H_1(A,B,B)=0$.
\end{Lemma}
For the proof note for example that $\Omega_{{\tilde B}/{\tilde A}} $ is the direct limit of $\Pi_{u\in {\mathcal U}}(\Omega_{B/A})_u$, where $(\Omega_{B/A})_u=\Omega_{B/A}$ (see the Appendix).

\begin{Proposition}\label{one}

Let $V\subset V'$ be an  extension of valuation rings of dimension one containing $\bf Q$ such that its residue field extension is trivial. Assume that the value groups  $\Gamma, \Gamma'\subset {\bf R}$ of $V$ respectively $V'$ are   dense in $\bf R$ and the factor of  the value groups  $\Gamma'/\Gamma\not =0$ has no torsion. 
Then the extension $V\subset V'$ is not ind-smooth.
\end{Proposition}
 
\begin{proof} Using Proposition \ref{pass-to-hat} we may suppose that $V,V'$ are complete. So $V$ contains its residue field.  By Variant \ref{variant} we find an extension of valuation rings ${\tilde V}\subset {\tilde V}'$ such that there exists  a cross-section ${\tilde s}: {\tilde \Gamma}'\to ({\tilde K}')^*$ such that ${\tilde s}({\tilde \Gamma})\subset {\tilde K }$. We remind that we wrote  $\Gamma'$ as a filtered  union of finitely generated subgroups $(\Gamma'_i)_{i\in I}$ and  set $\Gamma_i=\Gamma'_i\cap \Gamma$.
Set $V_{0i}=V\cap s'(\Gamma_i)=V'\cap s'(\Gamma_i)$,  $V'_{0i}=V'\cap s'(\Gamma'_i)$. By Remark \ref{rem} the modules $\Omega_{V'_{0i}/V_{0i}}$, $i \in I$ have torsion. Note that the filtered  union $V_{01}$ of  
$V_{0i}$, $i\in I$ is a valuation ring with value group $\Gamma$, and similarly consider $V'_{01}$ which has the value group $\Gamma'$.  Clearly, $\Omega_{V'_{01}/V_{01}}$ has torsion since it is the limit of $V'_{01}\otimes_{V'_{01}}\Omega_{V'_{0i}/V_{0i}}$. Set ${\tilde V}_0= {\tilde V}\cap s'({\tilde \Gamma})$,
${\tilde V}'_0= {\tilde V}'\cap s'({\tilde \Gamma}')$. By iteration we define the extensions $V_{0n}\subset V_n$ and $V'_{0n}\subset V'_n$ with the same  value group $\Gamma_n$, $\Gamma'_n$ obtained taking $n$-ultrapowers of $\Gamma$, resp. $\Gamma'$ and we see that $\Omega_{V'_{0n}/V_{0n}}$ has torsion by Lemma \ref{l 5.6}.  Then $\Omega_{{\tilde V}'_0/{\tilde V}_0}$ has torsion since it is the limit of   ${\tilde V'}_0\otimes_{V'_{0n}}\Omega_{V'_{0n}/V_{0n}}$.

Assume that $V\subset V'$ is ind-smooth.
 In the Jacobi-Zariski sequence 
applied to ${\tilde V}_0,{\tilde V},{\tilde V}'$ 
$$H_1({\tilde V},{\tilde V}',{\tilde V}')\to {\tilde V}'\otimes_{\tilde V} \Omega_{{\tilde V}/{\tilde V}_0}\to \Omega_{{\tilde V}'/{\tilde V}_0}\to \Omega_{{\tilde V}'/{\tilde V}} \to  0$$
we claim that the left module is zero  and the last module has no torsion by Lemmas \ref{Lemma 5.1},  \ref{l 5.6} because $\Omega_{V'/V} $ has no torsion and $H_1(V,V',V')=0$,  $V\subset V'$ being ind-smooth. More precisely, we see that $\Omega_{V'_n/V_n}$ has no torsion and $H_1(V_n,V'_n,V'_n)=0$ for all $n$ using Lemma \ref{Lemma 5.1} and  by iteration Lemma \ref{l 5.6}. At the limit we get our claim.

Also  $\Omega_{{\tilde V}/{\tilde V}_0}$ is flat (so it  has no torsion) since the extension  ${\tilde V}_0\subset {\tilde V}$ is ind-smooth having the same value group (see Proposition \ref{p}). It follows that $\Omega_{{\tilde V}'/{\tilde V}_0}$ has also no torsion.
 
 Now,  in the Jacobi-Zariski sequence 
applied to ${\tilde V}'_0,{\tilde V}'_0,{\tilde V}'$ 
$$H_1({\tilde V}_0,{\tilde V}',{\tilde V}')\to {\tilde V}'\otimes_{{\tilde V}'_0} \Omega_{{\tilde V}'_0/{\tilde V}_0}\to \Omega_{{\tilde V}'/{\tilde V}_0}$$
we see that the left module is zero by Lemma \ref{Lemma 5.1} 
 because ${\tilde V}'_0\subset {\tilde V}'$ is ind-smooth by Proposition \ref{p}. As above  $\Omega_{{\tilde V}'_0/{\tilde V}_0}$ has torsion and so $\Omega_{{\tilde V}'/{\tilde V}_0}$ has torsion too, which is false.
\hfill\ \end{proof}

{\bf Acknowledgements:}
This work has been partially elaborated in the frame of the International
Research Network ECO-Math.


\newpage

\section*{Appendix. Cross-sections via infinite towers of ultrapowers \\ by K\k{e}stutis \v{C}esnavi\v{c}ius\protect\footnote{CNRS, UMR 8628, Laboratoire de Math\'{e}matiques d'Orsay, Univ.~Paris-Sud, Universit\'{e} Paris-Saclay, 91405 Orsay, France. Email: {\sf kestutis@math.u-psud.fr} \\  \indent 
I thank Dorin Popescu for helpful discussions and for hosting me during a visit to Romania in March 2019. I thank Matthias Aschenbrenner for helpful comments.
This project has received funding from l'Agence Universitaire de la Francophonie and the European Research Council (ERC) under the European Union's Horizon 2020 research and innovation programme (grant agreement No. 851146).} }

The goal of this Appendix is to show that by replacing a valuation ring by the limit of an infinite tower of its suitable ultrapowers one may arrange the valuation map $\val \colon V \setminus \{ 0 \} \to \Gamma$ to admit a multiplicative section (see Theorem \ref{big-tower}). For this, we use techniques from model theory, specifically, the Keisler--Kunen theorem about the existence of good ultrafilters:\footnote{After this appendix was written, we learned of a much simpler way to deduce sharper versions of Theorem \ref{big-tower} and Variant \ref{variant} from the results reviewed in \S\ref{alg-comp}, see \cite[3.3.39, 3.3.40]{ADH17}. We left this appendix in place in case the method used here would prove useful for other purposes.}  the idea is that constructing a section amounts to solving a system of equations for which any finite subsystem has a solution, and such systems always have solutions in well-chosen ultrapowers. For instance, if the system is countable, then solutions exists in any nonprincipal ultrafilter on $\bf N$, and in general the main subtlety is in constructing the ultrafilter (within ZFC). 

\bbppt[Cross-sections of valuation rings]

For a valuation ring $V$ with the value group $\Gamma$ and the fraction field $K$, a \emph{cross-section} of $V$ is a section 
\[
s \colon \Gamma \to K^* \ \ \mbox{in the category of abelian groups to the valuation map} \ \ \val \colon K^* \to \Gamma.
\]
Concretely, $s$ is a group homomorphism such that $\val(s(\gamma)) = \gamma$ for $\gamma \in \Gamma$.  For a submonoid $M \subset \Gamma$, a \emph{partial cross-section} defined on $M$ is a monoid morphism $s \colon M \to K^*$ (concretely, $s(0) = 1$ and $s(m + m') = s(m)s(m')$) with $\val(s(m)) = m$ for $m \in M$. Evidently, partial cross-sections exist for $M \simeq {\bf Z}_{\geq 0}^r$ and correspond to choices of tuples of elements of $K^*$ whose valuations form the standard basis of ${\bf Z}_{\geq 0}^r$.
\eeppt

\begin{ExampleT}
Cross-sections exist when $\Gamma$ is free as a $\bf Z$-module, for instance, when it is finitely generated. As we now explain, they also exist when $V$ is strictly Henselian of residue characteristic $p$ and there is a free subgroup $\Gamma_0 \subset \Gamma$ such that $\Gamma/\Gamma_0$ is torsion with $(\Gamma/\Gamma_0)[p^\infty] = 0$. Indeed, we first define $s$ on $\Gamma_0$ and then, by Zorn's lemma, reduce to the situation in which $s$ is already defined on some subgroup $\Gamma' \supset \Gamma_0$ and needs to be extended to a $\Gamma'' \supsetneq \Gamma'$ with $\Gamma''/\Gamma'$ cyclic of order $n$ prime to $p$. For the latter, we first choose an $x \in V$ such that $\val(x)$ lies in $\Gamma''$ and generates the quotient $\Gamma''/\Gamma'$, which gives the following equation in $V$:
\[
x^n = u\cdot s(n\cdot \val(x)) \qxq{for some} u \in V^*.
\] 
Since $p \nmid n$, Hensel's lemma \cite[IV, 18.5.17]{EGA} (which is the Implicit Function Theorem in this context) now implies that the equation $X^n = u$ has a solution in $V$, so we may adjust $x$ to assume that $u = 1$. Granted this, $s$ then extends to $\Gamma''$ by setting $s(\val(x)) = x$: indeed, any relation $N \cdot \val(x) = \gamma$ with $N \in {\bf Z}$ and $\gamma \in \Gamma'$ must be a multiple of such a relation with $N = n$, so $s(N \cdot \val(x)) = s(\gamma)$. 
\end{ExampleT}

\bbppt[Ultrafilters and ultraproducts] \label{ultra-section}
We recall that an \emph{ultrafilter} on a nonempty set $U$ is a set $\sU$ of subsets of $U$ that is closed under finite intersections, closed under taking supersets, does not contain the empty set, and for every $U' \subset U$ contains either $U'$ or $U \setminus U'$. Such a $\sU$ is \emph{principal} if it consists of all the subsets containing some fixed $u \in U$, and is \emph{nonprincipal} otherwise. An ultrafilter $\sU$ is \emph{countably incomplete} if some countable collection of elements of $\sU$ has an empty intersection. Such a $\sU$ is also nonprincipal and it exists whenever $U$ is infinite (see \cite[\S A.3,~8.4]{Bar}). We view any ultrafilter $\sU$ as a partially ordered set, where $U' \le U''$ if $U' \supseteq U''$. 	 

For any category $\mathcal C$ that has small products and filtered direct limits, the \emph{ultraproduct} of a set $\{C_u \}_{u \in U}$ of objects of $\mathcal C$ with respect to an ultrafilter $\sU$ on $U$, which we denote abusively by $\prod_{\sU} C_u$, is
\[
\tst \varinjlim_{U' \in {\sU}} ( \prod_{u \in U'} C_u ) \qx{where transition maps are projections onto partial products}
\]
(the limit is filtered because $\sU$ is closed under finite intersections). In the case when all the $C_u$ are the same object $C \in {\mathcal C}$, we call $\prod_{\sU} C$ an \emph{ultrapower} of $C$. 
\eeppt


\bbppt[Ultraproducts of valuation rings] \label{ultra-val}
We will work with ultraproducts of rings or modules. For instance, an ultraproduct of fields is again a field: every nonzero element is invertible (thanks to the axiom that $U' \in {\sU}$ or $U \setminus U' \in {\sU}$). Likewise, an ultraproduct $\prod_{\sU} V_u$ of valuation rings $\{V_u\}_{u \in U}$ with fraction fields $\{K_u\}_{u \in U}$ is a valuation ring with fraction field $\prod_{\sU} K_u$: for any nonzero element $v$ of the latter, either $v$ or $v^{-1}$ lies in $\prod_{\sU} V_u$. We see similarly that
\begin{enumerate}
\item 
the maximal ideal of  $\prod_{\sU} V_u$ is the ultraproduct $\prod_{\sU} \mm_u$ of the maximal ideals; 

\item 
the residue field of  $\prod_{\sU} V_u$ is the ultraproduct $\prod_{\sU} k_u$ of the residue fields; 

\item 
the value group of $\prod_{\sU} V_u$ is the ultraproduct $\prod_{\sU} \Gamma_u$ of the value groups; 

\item 
the monoid of nonnegative elements $(\prod_{\sU} \Gamma_u)_{\geq 0}$  is identified with $\prod_{\sU} (\Gamma_u)_{\geq 0}$.
\end{enumerate}
\eeppt

\msk

The existence of ``well-chosen'' ultrapowers mentioned above rests on the Keisler--Kunen theorem from model theory that we recall in the following lemma. Keisler proved it in \cite{Kei} assuming the Generalized Continuum Hypothesis and Kunen gave an unconditional proof in \cite[Theorem 3.2]{Kun}.

\begin{LemmaT}[\cite{CK}, Theorem 6.1.4] \label{ultra-source}
For an infinite set $U$, there is a countably incomplete  ultrafilter $\sU$ on $U$ such that for any inclusion-reversing function 
\[
f \colon \{ \mbox{finite subsets of } U\} \to {\sU}, \ \ \mbox{there is a function} \ \ f_0 \colon \{ \mbox{finite subsets of } U\} \to {\sU}
\]
that is also inclusion-reversing 
and satisfies
\[
\ \ f_0(U') \subset f(U') \ \ \mbox{and} \ \ f_0(U' \cup U'') = f_0(U') \cap f_0(U'') \ \ \mbox{for all finite subsets} \ \ U', U'' \subset U.  
\]
\end{LemmaT}

Of course, the requirement that $f_0$ be inclusion-reversing is superfluous: it is a special case of the requirement that $f_0$ transform finite unions into intersections.

\bsk

We now verify that the ultrapowers that result from the ultrafilters supplied by Lemma \ref{ultra-source} have the promised property of solvability of systems of equations.

\begin{PropositionT} \label{CK-lemma}
For an infinite cardinal $\kappa$, every ultrafilter $\sU$ supplied by Lemma {\upshape\ref{ultra-source}} for a set $U$ of cardinality $\kappa$ is such that\ucolon for any ring $R$ \up{resp.,~and any left $R$-module $M$}, any polynomial \up{resp.,~linear} system of equations
\[
\tst \{ g_i( \{X_\sigma\}_\sigma) = 0 \}_{i \in I} \qxq{{\upshape(}resp.,}  \{ \sum_{\sigma} r_{i,\, \sigma} X_{\sigma} = m_i \}_{i \in I}) \ \ \mbox{with} \ \ \# I \leq \kappa
\]
in variables $\{X_\sigma\}_\sigma$ and coefficients in $\prod_{\sU} R$ \up{resp.,~$r_{i,\, \sigma} \in \prod_{\sU }R$ and $m_i\in \prod_{\sU} M$} has a solution in $\prod_{\sU} R$ \up{resp.,~$\prod_{\sU} M$} as soon as so do all its finite subsystems.
\end{PropositionT}

\begin{proof}
The assertion is a concrete case of the model-theoretic \cite[Theorem 6.1.8]{CK}, and the latter is sharper in multiple aspects. For convenience, we recall the argument. 

For brevity, we denote the system in question by $\{g_i = 0\}_{i \in I}$ and we lift it to a system $\{\wt{g}_i = 0\}_{i \in I}$ with coefficients in $\prod_{u \in U} R$ (resp.,~and in $\prod_{u \in U} M$) and the same variables $\{X_\sigma\}_{\sigma}$ by lifting the nonzero coefficients along the surjection 
\[
\tst \prod_{u \in U} R \surjects \prod_{\sU} R \qxq{\upshape{(}resp.,~and along} \prod_{u \in U} M \surjects \prod_{\sU} M \qx{for the $m_i$}).
\]
Since $\sU$ is countably incomplete, we may fix a decreasing sequence 
\[
\tst  U \supset U_0 \supset U_1 \supset U_2 \supset \cdots \qxq{of sets in $\sU$ with}  \bigcap_{n \geq 0} U_n = \emptyset.
\]
We then define a function
\[ 
f \colon \{\mbox{finite subsets of $I$}\} \to {\sU} 
 \]
by letting $\mathrm{pr}_u$ denote the projection onto the $u$-th factor of $\prod_{u \in U}$ and setting
\[
 f(I') := U_{\# I'} \cap \{ u \in U\, \vert\, \mbox{the system } \{ \mathrm{pr}_u(\wt{g}_i) = 0 \}_{i \in I'} \mbox{ is solvable in $R$ (resp.,~$M$)} \}.
 \]
The well-definedness of $f$ follows from the solvability of the subsystem $\{g_i = 0\}_{i \in I'}$ in $\prod_{\sU} R$ (resp.,~$\prod_{\sU} M$) and from the stability of $\sU$ under supersets. By construction, $f(I') \supset f(I'')$ whenever $I' \subset I''$, so, since $\#I \le \# U$, Lemma \ref{ultra-source} supplies a function
\[ 
f_0 \colon \{\mbox{finite subsets of $I$}\} \to {\sU} \ \ \mbox{such that} \ \ f_0(I') \subset f(I'),  \ f_0(I' \cup I'') = f_0(I') \cap f_0(I'')
 \]
for all finite subsets $I', I'' \subset I$ (technically, to apply Lemma \ref{ultra-source} we first embed $I$ into $U$ as a subset and then extend $f$ to finite subsets $U' \subset U$ by the rule $U' \mapsto f(U' \cap I)$).

For each $u \in U$, we set
\[
I_u := \{ i \in I \, \vert\, u \in f_0(\{ i \}) \}.
\]
Whenever, $i_1, \dots, i_n \in I_u$ are pairwise distinct, we have
\[
u \in f_0(\{i_1\}) \cap \dots \cap f_0(\{i_n\}) = f_0(\{i_1, \dots, i_n\}) \subset f(\{i_1, \dots, i_n\}) \subset U_n,
\]
so, since the $U_n$ have empty intersection, each $I_u$ is finite. Then the preceding display applied to an enumeration of $I_u$ shows that $u \in f(I_u)$, to the effect that the system 
$
\{ \mathrm{pr}_u(\wt{g}_i) = 0 \}_{i \in I_u} \ \ \mbox{has a solution} \ \ \{x_{\sigma,\, u} \}_\sigma \ \ \mbox{in $R$ (resp.,~$M$)}.$

We claim that $\{ x_\sigma := (x_{\sigma,\, u})_{u \in U} \}_\sigma$ gives a solution in $\prod_{\sU} R$ (resp.,~$\prod_{\sU} M$) to the system $\{ g_i = 0\}_{i \in I}$. Indeed, for every $i \in I$ we have $f_0(\{ i \}) \in {\sU}$ and for every $u \in f_0(\{i \})$ we have $i \in I_u$, so $\wt{g}_i(\{ x_\sigma\}_\sigma) = 0$ in the projection on $\prod_{u \in f_0(\{i \})}$.
\hfill\ \end{proof}

The argument is not specific to rings or modules, and it also shows the following.

\begin{variantT} \label{monoid-var}
For an infinite cardinal $\kappa$, every ultrafilter $\sU$ supplied by Lemma {\upshape\ref{ultra-source}} for a set $U$ of cardinality $\kappa$ is such that\ucolon for any monoid $G$, any system 
\[
\{ g_i(\{X_\sigma\}_\sigma) = g_i'(\{X_\sigma\}_\sigma)\}_{i \in I} \qxq{with} \# I \le \kappa
\]
of monomial equations in variables $\{X_\sigma\}_\sigma$ and coefficients in $\prod_\sU G$ has a solution in $\prod_\sU G$ as soon as so does each of its finite subsystems. 
\end{variantT}

As we now review, Proposition \ref{CK-lemma} supplies algebraically compact ultrapowers.

\bbppt[Algebraic compactness] \label{alg-comp}
We fix a ring $R$ and recall that a map $M \to M'$ of left $R$-modules is \emph{pure} if the map $M'' \tensor_R M \to M'' \tensor_R M'$ is injective for every right $R$-module $M''$. An $R$-module $M$ is \emph{algebraically compact} (or \emph{pure-injective}) if every pure map $M \to M'$ of $R$-modules is a split injection. For example, if $M$ is an algebraically compact abelian group (so $R ={ \bf Z}$), then every short exact sequence 
\[
0 \to M \to M' \to M'/M \to 0 \qxq{of abelian groups with}  (M'/M)_{\mathrm{tors}} = 0 \qx{splits.}
\]
A concrete criterion for algebraic compactness is given by \cite[7.1 (with 6.5)]{JL}: a left $R$-module  $M$ is algebraically compact if every system of equations 
\[
\tst \{\sum_{\sigma} r_{i,\, \sigma} X_{\sigma} = m_i\}_{i \in I} \qxq{with} r_{i,\, \sigma} \in R \qxq{and} m_i \in M
\]
has a solution in $M$ as soon as so do all its finite subsystems. Moreover, by \cite[7.28, 7.29]{JL}, it suffices to consider systems with cardinality $\# I \le \max(\#R, \#{\bf Z})$. In particular, thanks to Proposition \ref{CK-lemma}, there is an ultrafilter $\sU$ such that for any $R$-module $M$, the $R$-module $\prod_{\sU} M$ is algebraically compact.
\eeppt
\ssk

With model-theoretic input in place, we turn to the tower of ultrapowers argument in Theorem \ref{big-tower}. The final input is the following lemma proved in \cite[2.2]{Ell}, \cite[4.6.1]{Po1}, or \cite[6.1.30]{GR} that captures the ``combinatorial'' part of local uniformization.

\begin{LemmaT} \label{EPGR-lem-re}
For a totally ordered abelian group $\Gamma$, the submonoid $\Gamma_{\ge 0} \subset \Gamma$ of nonnegative elements is a filtered increasing union of its finite free submonoids isomorphic to ${\bf Z}_{\ge 0}^r$ \up{where $r \in {\bf Z}_{\ge 0}$ need not be constant}. 
\end{LemmaT}

\begin{TheoremT} \label{big-tower}
For a valuation ring $V$ with value group $\Gamma$, there is a countable sequence of ultrafilters ${\sU}_1, {\sU}_2, \dots$ on some respective sets $U_1, U_2, \dots$ for which the valuation rings $\{V_n\}_{n \geq 0}$ defined inductively by $V_0 := V$ and $V_{n + 1} := \prod_{{\sU}_{n + 1}} V_{n}$ 
are such that the valuation ring
\[
\tst \wt{V} := \varinjlim_{n \geq 0} V_n \qxq{has a cross-section}  \wt{s} : \wt{\Gamma} \to \wt{K}^*,
\]
where $\wt{K}$ and $\wt{\Gamma}$ are the fraction field and the value group  of $\wt{V}$.
\end{TheoremT}

\begin{proof}
We let $K_n$ and $\Gamma_n$ denote the fraction field and the value group of $V_n$, so that $\Gamma_{n + 1} \cong \prod_{{\sU}_{n + 1}} \Gamma_{n}$ and $K_{n + 1} = \prod_{{\sU}_{n + 1}} K_{n }$ (see \S\ref{ultra-val}) with
\[
\tst \wt{\Gamma} \cong \varinjlim_{n \geq 0} \Gamma_n \ \ \mbox{and} \ \ \wt{K} \cong \varinjlim_{n \geq 0} K_n.
\]
The idea is to build ultrafilters ${\sU}_n$ one by one using Lemma \ref{ultra-source} in such a way that a desired cross-section
\[
\tst \wt{s} \colon \wt{\Gamma} \to \wt{K}^* \ \ \mbox{would be the limit of compatible partial cross-sections} \ \ s_n \colon \Gamma_n \to K_{n + 1}^*.
\]
For this, as an initial step, we replace $V$ by a suitable ultrapower to ensure that the abelian group $\Gamma$ is algebraically compact (see \S\ref{alg-comp}). 
Granted this, it suffices to carry out the inductive step: setting $\Gamma_{-1} := 0$ for convenience and assuming that we have already constructed $s_{n - 1}$ and $V_n$ for some $n \geq 0$ in such a way that the abelian groups $\Gamma_{n - 1}$ and $\Gamma_n$ are algebraically compact, it suffices to construct $V_{n + 1}$ with $\Gamma_{n + 1}$ algebraically compact in such a way that $s_{n - 1}$ extends to an $s_n$. 

The role of algebraic compactness is to split the map $\Gamma_{n - 1} \hra \Gamma_n \cong \prod_{{\sU}_n} \Gamma_{n - 1}$ whose cokernel is torsion free:
\[
\Gamma_n \cong \Gamma_{n - 1} \oplus G \ \ \mbox{for some subgroup} \ \ G \subset \Gamma_n.
\]
 Thanks to this splitting, we only need to build an ultrafilter ${\sU}_{n + 1}$ and a partial cross-section $s_G : G \to (\prod_{{\sU}_{n + 1}} K_n)^*$ such that $\prod_{{\sU}_{n + 1}} \Gamma_n$ is algebraically compact. In fact, we let ${\sU}_{n + 1}$ be any ultrafilter as in Lemma \ref{ultra-source} applied to the cardinal $\max(\#\Gamma_n, \#{\bf Z})$. Then $\prod_{{\sU}_{n + 1}} \Gamma_n$ is necessarily algebraically compact by the criterion reviewed in \S\ref{alg-comp} and Proposition \ref{CK-lemma}.

The subgroup $G$ inherits a total order from $\Gamma_n$, and any partial cross-section 
\[
\tst  s_{G_{\geq 0}} : G_{\geq 0} \to (\prod_{{\sU}_{n + 1}} V_n) \setminus \{ 0 \}  \qxq{will give rise to a desired} s_G.
\]
For each $g \in G_{> 0}$, we fix a $v_g \in V_n$ with $\val(v_g) = g$. Then $s_{G_{\geq 0}}$ amounts to a solution in $\prod_{{\sU}_{n + 1}} V_n$ of the following system of equations in variables $\{X_g, U_g, U'_g\}_{g \in G_{> 0}}$:
\[
\tst  \{X_{g + g'} = X_g X_{g'}, \ \ X_gU_g = v_g, \ \ U_gU'_g = 1\}_{g,\, g' \in G_{> 0}}.
\]
Likewise, for any submonoid $G' \subset G_{\geq 0}$, the restriction of $s_{G_{\geq 0}}|_{G'}$, that is, a partial cross-section defined on $G'$, amounts to a solution in $\prod_{{\sU}_{n + 1}} V_n$ of the subsystem consisting of those equations that only involve the variables $\{ X_g, U_g, U'_g\}_{g \in G'}$. However, a partial cross-section $G' \to \prod_{{\sU}_{n + 1}} V_n$ (and even $G' \to V_n$) certainly exists if $G' \simeq {\bf Z}_{\geq 0}^d$, and, by Lemma \ref{EPGR-lem-re}, the monoid $G_{\geq 0}$ is a filtered increasing union of such $G'$. This implies that every finite subsystem of the above system  has a solution in $\prod_{{\sU}_{n + 1}} V_n$ (and even in $V_n$). Then, by Proposition \ref{CK-lemma}, the entire system  has a solution in $\prod_{{\sU}_{n + 1}} V_n$, which completes the inductive step.
\hfill\ \end{proof}

\begin{variantT} \label{variant}
For every faithfully flat map $V \subset V'$ of valuation rings  with value groups $\Gamma \subset \Gamma'$ such that $\Gamma'/\Gamma$ is torsion free, there is a countable sequence of ultrafilters ${\sU}_1, {\sU}_2, \dots$ on some respective sets $U_1, U_2, \dots$ for which the valuation rings $\{V_n\}_{n \geq 0}$ and $\{V'_n\}_{n \geq 0}$ defined inductively by $V_0 := V$ and $V'_0 := V'$ with $V_{n + 1} := \prod_{{\sU}_{n + 1}} V_{n}$ and $V'_{n + 1} := \prod_{{\sU}_{n + 1}} V'_{n}$ are such that the valuation ring 
\[
\tst  \wt{V}' = \varinjlim_{n \geq 0} V'_n \qxq{with} \wt{K}' := \Frac(\wt{V}') \qxq{has a cross-section}  \wt{s} :\wt{\Gamma}' \to \wt{K}'^*
\]
whose restriction to the value group $\wt{\Gamma}$ of $\wt{V} := \varinjlim_{n \geq 0} V_n$ lands in $\wt{K} := \Frac(\wt{V})$. 
\end{variantT}

\begin{proof}
An ultrapower of an ultrapower is itself an ultrapower \cite[6.5.2]{CK}, so we may
make an initial replacement of $V$ and $V'$ by suitable large ultrapowers and use §A.8
with Proposition A.6 to ensure that $\Gamma$  is algebraically compact and later absorb
the appearing initial ultrafilter into $\mathcal{U}_1$ (alternatively, we could simply insert this
initial ultrafilter as  $\mathcal{U}_1$  without using loc. cit.). Then, thanks to the torsion-freeness
assumption on $\Gamma'/\Gamma$, the inclusion $\Gamma\subset \Gamma'$ splits. A choice of a splitting induces a
compatible splitting on any ultrapower, so the proof of Theorem A.10 continues to
give the claimed variant granted that we take advantage of the splitting to build the
cross-section in such a way that its restriction to $\tilde \Gamma$
lands in $\tilde K$.
\hfill\ \end{proof}

\end{document}